\documentclass[11pt]{article}
\usepackage{amsmath,amssymb,amsthm,bbm,srcltx,graphicx,color,a4wide,xr, setspace}
\usepackage[colorlinks]{hyperref}
\usepackage[shortlabels]{enumitem}
\usepackage{tikz}
\usepackage{tikz}
\usetikzlibrary{arrows}
\numberwithin{equation}{section}
\newcommand{\bonote}[1]{\marginpar{\textcolor{magenta}{\small{Bo:\;#1}}}}

\newcommand{\dto}{\stackrel{d}{\longrightarrow}}

\newcommand{\fidi}{\stackrel{\mathrm{fi.di.}}{\longrightarrow}}
\newcommand{\toi}{\to\infty}

\def\nn{\nonumber}
\def\({\left(}
\def\){\right)}

\newcommand\1[1]{\mathds{1}_{\{#1\}}}

\newcommand{\EE}{\mathbb{E}}
\newcommand{\NN}{\mathbb{N}}
\newcommand{\cN}{{\cal N}}
\newcommand{\N}{\mathbb{N}}
\newcommand{\PP}{\mathbb{P}} 
\newcommand{\RR}{\mathbb{R}} 
 
\newcommand{\ZZ}{\mathbb{Z}} 
\newcommand{\calC}{\mathcal{C}}

\newcommand{\cc}{\mathsf{c}}

\newcommand{\ff}{\infty} 
\newcommand{\vep}{\varepsilon}
\newcommand{\tvep}{\tilde\varepsilon}

\newcommand{\law}{\mathcal{L}}

\newcommand{\bsH}{\boldsymbol{H}}

\newcommand{\bsX}{\boldsymbol{X}}
\newcommand{\bsY}{\boldsymbol{Y}}
\newcommand{\bsZ}{\boldsymbol{Z}}
\newcommand{\bsQ}{\boldsymbol{Q}}

\newcommand{\hA}{\boldsymbol{\hat{A}}}
\newcommand{\bsA}{\boldsymbol{A}}
\newcommand{\hB}{\boldsymbol{\hat{B}}}

\newcommand{\ccC}{\boldsymbol{{\cal C}}}
\newcommand{\bsB}{\boldsymbol{B}}

\newcommand{\bsE}{\boldsymbol{E}}
\newcommand{\bsO}{\boldsymbol{0}}

\newcommand{\by}{\boldsymbol{y}}
\newcommand{\bx}{\boldsymbol{x}}

\newcommand{\tilE}{\widetilde{E}}
\newcommand{\VT}{\, \vert \,}

\newcommand{\dsum} {\displaystyle\sum}

\newcommand{\lo}{\tilde{l}_0}
\newcommand{\lom}{\tilde{l}_{0,m}}

\def\PPC{N}

\def\t{\tilde}
\def\h{\hat}

\def\iid{i.i.d.}

\usepackage{cleveref}

\newtheorem{theorem}{Theorem}[section]
\newtheorem{lemma}[theorem]{Lemma}
\newtheorem{corollary}[theorem]{Corollary}
\newtheorem{proposition}[theorem]{Proposition}

\theoremstyle{remark}
\newtheorem{remark}[theorem]{Remark}
\newtheorem{example}[theorem]{Example}

\crefname{hypothesis}{Assumption}{Assumptions}

\definecolor{mygray}{gray}{0.6}
\usepackage{comment}
\usepackage{ifthen}
\usepackage{dsfont}
\usepackage{hyperref}
\hypersetup{
urlcolor=black, 
  menucolor=black, 
  citecolor=black, 
  anchorcolor=black, 
  filecolor=black, 
  linkcolor=black, 
  colorlinks=true,
}
\definecolor{darkblue}{rgb}{.1, 0.1,.8}
\definecolor{darkgreen}{rgb}{0,0.8,0.2}
\definecolor{darkred}{rgb}{.8, .1,.1}
\newcommand{\red}{\color{darkred}}

\newcommand{\cid}{\stackrel{d}{\rightarrow}}
\newcommand{\cip}{\stackrel{\P}{\rightarrow}}

\newcommand{\hX} {\hat{\X}}

\newcommand{\bth}{\begin{theorem}}
\newcommand{\ethe}{\end{theorem}}

\newcommand{\bre}{\begin{remark}\em }
\newcommand{\ere}{\end{remark}}

\newcommand{\ble}{\begin{lemma}}
\newcommand{\ele}{\end{lemma}}

\newcommand{\bde}{\begin{definition}}
\newcommand{\ede}{\end{definition}}
\newcommand{\bco}{\begin{corollary}}
\newcommand{\eco}{\end{corollary}}

\newcommand{\bpr}{\begin{proposition}}
\newcommand{\epr}{\end{proposition}}

\newcommand{\bexer}{\begin{exercise}}
\newcommand{\eexer}{\end{exercise}}

\newcommand{\bexam}{\begin{example}}
\newcommand{\eexam}{\end{example}}

\newcommand{\bfi}{\begin{fig}}
\newcommand{\efi}{\end{fig}}

\newcommand{\btab}{\begin{tab}}
\newcommand{\etab}{\end{tab}}

\newcommand{\rhs}{right-hand side}

\newcommand{\beao}{\begin{eqnarray*}}
\newcommand{\eeao}{\end{eqnarray*}\noindent}

\newcommand{\beam}{\begin{eqnarray}}
\newcommand{\eeam}{\end{eqnarray}\noindent}

\newcommand{\beqq}{\begin{equation}}
\newcommand{\eeqq}{\end{equation}\noindent}

\newcommand{\bce}{\begin{center}}
\newcommand{\ece}{\end{center}}

\newcommand{\barr}{\begin{array}}
\newcommand{\earr}{\end{array}}

\newcommand{\vague}{\stackrel{\lower0.2ex\hbox{$\scriptscriptstyle
                    \it{v} $}}{\rightarrow}}
\newcommand{\weak}{\stackrel{\lower0.2ex\hbox{$\scriptscriptstyle
                    \it{w} $}}{\rightarrow}}
\newcommand{\what}{\stackrel{\lower0.2ex\hbox{$\scriptscriptstyle
                    \it{\hat{w}} $}}{\rightarrow}}

\newcommand{\bdis}{\begin{displaymath}}
\newcommand{\edis}{\end{displaymath}\noindent}

\renewcommand{\P}{\mathbb P}

\newcommand{\R}{\mathbb{R}}

\newcommand{\nto}{n\to\infty}

\newcommand{\la}{\lambda}

\newcommand{\bfS}{{\bf S}}

\def\E{{\mathbb E}}

\def\P{{\mathbb{P}}}
\def\R{\mathbb{R}}
\def\Z{\mathbb{Z}}

\renewcommand{\le}{\ensuremath{\leqslant}}
\renewcommand{\leq}{\ensuremath{\leqslant}} 
\renewcommand{\geq}{\ensuremath{\geqslant}} 
\renewcommand{\ge}{\ensuremath{\geqslant}}

\def\iid{independent and identically distributed}
\newcommand{\indicator}{\mathds{1}}

\newcommand{\bfH}{{\mathbf H}}

\newcommand{\bfQ}{{\mathbf Q}}

\renewcommand{\P}{\mathbb{P}}

\newcommand{\Frechet}{Fr\'{e}chet }

\DeclareMathOperator{\e}{e}

\newcommand{\X}{{\mathbf X}}

\newcommand{\A}{{\mathbf A}}

\newcommand{\norm}[1]{\|#1\|}

\newcommand{\eid}{\stackrel{\rm d}{=}}

\begin{document}

	\title{Extreme eigenvalue statistics of $m$-dependent heavy-tailed matrices} 
	
	\author{Bojan Basrak\thanks{Department of Mathematics, University of Zagreb, Bijeni\v cka 30,
			Zagreb, Croatia; {\tt bbasrak@math.hr }} \and
		Yeonok Cho\thanks{Department of Mathematical Sciences, KAIST, Daejeon, South Korea; {\tt imyo@kaist.ac.kr}}
		\and
		Johannes Heiny\thanks{Department of Mathematics, Ruhr-University Bochum, Universit\"atsstra{\ss}e 150, Bochum, Germany; {\tt johannes.heiny@rub.de}}
		\and
		Paul Jung\thanks{Department of Mathematical Sciences, KAIST, Daejeon, South Korea; {\tt pauljung@kaist.ac.kr}}}
	
	\maketitle
	
	\begin{abstract}
		{
		We analyze the largest eigenvalue statistics of $m$-dependent heavy-tailed Wigner matrices as well as the associated sample covariance matrices having entry-wise regularly varying tail distributions with parameter $\alpha\in (0, 4)$. Our analysis extends results in the previous literature for the corresponding random matrices with independent entries above the diagonal, by allowing for $m$-dependence between the entries of a given matrix. We prove that the limiting point process of extreme eigenvalues is a Poisson cluster process. \vspace*{0.3cm}

Nous analysons les plus grandes valeurs propres d'une matrice de Wigner avec entrées $m$-dépendantes et à queue lourde, de même que pour une matrice de covariance  associée avec entrées de variation régulière de paramètre $\alpha\in (0, 4)$. Notre analyse étend les résultats existants  pour ces matrices aléatoires avec entrées indépendantes à des entrées $m$-dépendantes.  Nous prouvons que le processus ponctuel limite des plus grandes valeurs propres est un processus de Poisson groupé.}
		
		
	\end{abstract}
	

	\section{Introduction}	

The phenomenon of universality for extreme eigenvalues of $n\times n$ symmetric random matrices $\hX$ (in this paper, the ``hat'' will denote a {\it symmetric} matrix), with i.i.d.\ real-valued entries on and above the diagonal, dates back to the seminal work of Soshnikov in \cite{soshnikov1999universality}. In that paper, and in several subsequent papers by varying authors culminating in \cite{lee2014}, it was shown that, as $n$ tends to infinity, {the distribution of the properly normalized largest eigenvalue of these matrices $\hX$, which are called Wigner matrices in the random matrix literature, converges to a Tracy--Widom distribution with parameter $1$, if and only if the matrix entries (with generic element $X$) satisfy
\begin{equation}\label{cond1234}
\lim_{x \to \infty} x^4 \P(|X|>x) = 0\,.
\end{equation}
 In particular, universality means that, as long as \eqref{cond1234} holds, the asymptotic fluctuations of the largest eigenvalues do not depend on the entry distribution. In this sense, the behavior of the largest eigenvalues (sometimes referred to as the right edge of the spectrum of $\hX$) is universal.}

	For Wigner matrices with heavy-tailed entries (in this paper, this will mean $\E[X^4]=\infty$), the behavior of the largest eigenvalues is not universal in general. 
For regularly varying entry distributions with exponent $\alpha\in(0,2)$, i.e.,  
	\begin{equation}\label{regvar}
\P(|X|>x) = x^{-\alpha} \widetilde{L}(x) 
\end{equation}
and $\widetilde{L}$ being a slowly varying function (at infinity), Soshnikov \cite{soshnikov:2004} showed that the limit of the largest eigenvalues depends on $\alpha$. More precisely, he proved that the point processes of properly normalized positive eigenvalues of the Wigner matrices converge in distribution to a Poisson point process on $(0,\ff)$ with 
	intensity  $\alpha x^{-1-\alpha}$. Auffinger et al.~\cite{auffinger:arous:peche:2009} extended this result to regularly varying entry distributions with index $\alpha\in[2,4)$. Similar results were obtained under additional assumptions such as sparsity and band structure; see for example \cite{benaych:peche:2014, auffinger:tang:2016}.

Regarding heavy-tailed sample covariance matrices of the form $\X\X'$, \cite{auffinger:arous:peche:2009} derived the limiting point process of suitably normalized eigenvalues in the case of {i.i.d.} entries.  By employing a large deviations approach, \cite{heiny:mikosch:2017:iid} allowed for more general growth rates of the dimension with respect to the sample size. 
When the entries of $\X$ are linear processes in space and time, \cite{davis2016extreme} have shown that the point process of eigenvalues converges to a Poisson cluster process; see also \cite{davis:mikosch:pfaffel:2016,davis:pfaffel:stelzer:2014, heiny:mikosch:2019:stochvol} for similar results.

As seen in \cite{auffinger:arous:peche:2009}, the properly normalized largest eigenvalues of Wigner and sample covariance matrices with regularly varying entries and $\alpha\in (0,4)$ are asymptotically \Frechet distributed with parameters $\alpha$ and $\alpha/2$, respectively. At $\alpha=4$ a phase transition occurs.

With regards to the empirical spectral distributions (e.s.d.) of $\hX$, which is defined as the distribution with point masses $1/n$ at every eigenvalue of $\hX$, the critical exponent of the tail is $2$ rather than $4$. For $\alpha>2$, it is well known that the e.s.d. of $\h{\bsX}/\sqrt{n}$ converges to the semi-circle distribution \cite{anderson2010introduction}. In contrast, if the entries are regularly varying with $\alpha\in(0,2)$, \cite{benarous2008spectrum, bordenave2011spectrum} have shown that the e.s.d. of suitably normalized $\h{\bsX}$ converges to a heavy-tailed probability measure with index $\alpha$. In the critical case $\alpha=2$, the e.s.d. converges to the semi-circle distribution \cite{jung2016levy}. The critical tail exponent for the e.s.d. of $\X\X'/n$ is also $2$  with the Marchenko-Pastur distribution taking the place of the semi-circle distribution for $\alpha>2$ and limiting e.s.d. which is heavy-tailed when $\alpha<2$  \cite{belinschi2009spectral}.

\subsection{Model and notation}\label{sec:model}	 	

While most of the arguments in this work can be adapted to matrices with complex-valued entries, the applications we are aiming for concern matrices with real entries, thus for simplicity we will assume all random variables in this work to be real-valued. We say that a random field $(X_{ij})_{(i,j) \in \mathcal{M}^2}$ with $\mathcal{M}\subseteq \Z$ is  $m$-dependent if for any subsets $A,B \subseteq \mathcal{M}^2$ with the property that $\max\{|i-k|,|j-\ell|\}>m$, whenever $(i,j) \in A, \, 
	(k,l)\in B$, the families of random variables $(X_{ij})_{(i,j)\in A}$ and  $(X_{kl})_{(k,l)\in B}$ are independent.
	
Throughout this paper we consider a stationary $m$-dependent random field $\bsX^{\infty}=(X_{ij})_{i,j \ge 1}$ with generic entry $X$, i.e., $X\eid X_{11}$, where $\eid$ denotes equality in distribution. 
	We impose the regular variation condition \eqref{regvar} with $\alpha \in (0,4)$. Choose a sequence $(a_n)$ such that
\begin{equation}\label{eq:defan}
n\, \P(|X| > a_n )\to 1\,,\qquad \nto\,.
\end{equation}
It is well known that $a_n=n^{1/\alpha} \ell(n)$, where $\ell$ is some slowly varying function.

Next, we reflect the upper triangular array $(X_{ij})_{i\le j}$ over the diagonal to obtain a symmetrized field $\h\bsX^{\infty}=(\h X_{ij})_{i,j\ge 1}$, where $\h X_{ij}=\h X_{ji}$ (the ``hat'' denotes the symmetrization via reflection). In the sequel, a boldface uppercase variable represents a doubly-indexed array of random variables -- in particular random matrices and/or random fields.

We define the Hermitian random matrices
\begin{equation}\label{eq:defwigner}
\hA=\hA_n= (\h X_{ij}/a_{n^2})_{1\le i,j\le n}\,, \qquad n\ge 1\,.
\end{equation}
Note that if $\bsX^{\infty}$ is an {i.i.d.} field, then $\hA$ is a (classical) Wigner matrix. For simplicity, we will also refer to $\hA$ as Wigner matrix if $\bsX^{\infty}$ is not an {i.i.d.} field.

For an integer sequence $p=p_n$ satisfying $p/n\to\gamma\in(0,\infty)$, we consider the data matrices
\begin{equation}\label{eq:defdata}
\A=\A_n= (X_{ij}/a_{np})_{1\le i\le p; 1\le j\le n}\,, \qquad n\ge 1\,.
\end{equation} 
and form the $p\times p$ {\em sample covariance matrices} $\A \A'$.

For any Hermitian matrix $\bfH$ we denote its ordered eigenvalues and singular values by $\la_1(\bfH) \ge \la_2(\bfH)\ge \cdots$ and $\sigma_1(\bfH) \ge \sigma_2(\bfH)\ge \cdots$, respectively. The spectral norm of a matrix $\bfH$ is defined as $\norm{\bfH}:= \sigma_1(\bfH)=\sqrt{\la_1(\bfH \bfH')}$.

\subsection{Objective and structure of this paper}
 
	In this work, we prove a certain universality of limiting extreme eigenvalues for heavy-tailed random matrices with $m$-dependent entries. To our knowledge, these are the first results regarding edge-universality of random matrices in the {general} $m$-dependent case. As we will see, the dependence between entries complicates the analysis considerably; however, the heavy-tailed condition \eqref{regvar} on the entries will allow us  to get a handle on the dependence. One motivating example of dependence between matrix elements, particularly in the heavy-tailed case, are (squared) sample covariance matrices of log-returns for the S\&P 500. In this case, one should not expect the returns of stocks in the same sector to be independent from each other, and the parameter $m$ can be thought of as being related to the number of stocks in a sector. We refer to \cite{davis2016extreme}  for more details, supporting data, and other motivating examples of dependent random matrices.
	
The main result of this paper roughly stated is, as $n\to\ff$:	
	\begin{align*}
	&\textit{The limiting largest eigenvalues of}  \  (\hA)_{n\in\N}  \ \textit{ and } \ (\A\A')_{n\in\N} \textit{ converge}\\
	&\textit{in distribution to the largest points of certain Poisson cluster processes.}
	\end{align*} 
	{Poisson cluster processes can be viewed as marked Poisson processes with markings which are clusters of points.}
	In order to state a precise detailed form and structure of the Poisson cluster processes in the above statement (as well as a precise definition of such processes), 
	we will first need to present some theory on regularly varying random fields.
	This will be done in  Section \ref{sec:random fields} (this section can be skimmed on a first reading, and referred back to as needed). After presenting this theory of regularly varying fields we will be able to state, in Section \ref{sec:refl}, a precise version of the above result concerning the largest eigenvalue statistics for a sequence of $m$-dependent heavy-tailed Wigner random matrices. In Section \ref{sec:cov}, we will state an analogous result for a sequence of $m$-dependent heavy-tailed sample covariance matrices. Before getting too technical, in Section \ref{sec:example},
	 we present motivating examples of $m$-dependent matrix toy models. 
	 Finally, in Sections \ref{sec:proofwigner} and \ref{sec:proofcov} we provide the proofs of our main results, Theorems \ref{tm:ReflMat} and \ref{tm:CovMat}, respectively.

We wrap up this introduction with a very brief top-level overview of our proof strategy.
The basic idea of \cite{soshnikov:2004} (and later \cite{auffinger:arous:peche:2009}) was to show that the extreme values of the independent matrix entries in the upper triangle are asymptotically equal in distribution to the extreme eigenvalues.
 Since we consider matrices where the entries are $m$-dependent, rather than independent, 
 an excessively large entry can affect several nearby entries and hence the maximal eigenvalues cannot simply be approximated by the extreme values of entries. To handle this problem, we employ multi-scale analysis and decompose our matrix into $k_{n}^2$ blocks of the same size $r_{n}\times r_{n}$ such that $k_{n},r_{n}\to\infty$ with some specified orders. Since the size of each block increases to infinity, dependence between the matrix entries stays within each block (except at the edges of the blocks, which are negligible), asymptotically. Decomposing our matrix into blocks, we now can compare the behavior of extreme eigenvalues to the behavior of the ``extreme blocks''. For the distribution within each block, we rely on results of Basrak, Planini\'{c}, and Soulier from \cite{basrak2016invariance}, in which the authors studied the limiting behavior of $m$-dependent stationary and regularly varying random fields (this is described in Section \ref{sec:random fields}).

In order to handle the dependence between different blocks of our random matrices at the level of the spectrum, we adapt an approach introduced in \cite{auffinger:arous:peche:2009}.  Roughly, we truncate the matrices by removing all small entries, so that, with high probability, only one block remains in any row of blocks or column of blocks after the truncation. Then we bound the operator norm of the truncated portion, so that the contribution of truncated blocks towards the spectrum, as well as their effect on the remaining blocks containing large entries,  is negligible by an application of Weyl's eigenvalue-perturbation inequality. We remark that showing that the operator norm of the truncated portion is negligible, even when replacing independence with $m$-dependence, is perhaps the technically most difficult portion of the overall proof-- this is presented in Section \ref{sec:proofwigner} which comprises the main mathematical contributions of this paper.

	\section{Background: regularly varying random fields}\label{sec:random fields}
A stationary random field $\bsX^{\infty}=(X_{ij})_{i,j\in \NN}$ is regularly varying if all the finite-dimensional vectors
	are regularly varying, see \cite{davis:hsing:1995} for
	instance. 	Recall that a $d$-dimensional random vector
	$\vec{V}$ is regularly varying with index $\alpha > 0$ if there exists a random vector
	$\vec{\Theta}$ on the unit sphere in $\mathbb{R}^{d}$ such that
	\begin{equation}
	\label{def:RV_d}
	\PP(\|\vec{V}\| > ux, \vec{V} / \|\vec{V}\| \in \cdot )/
	{\PP(\|\vec{V}\| > x)} \Rightarrow u^{-\alpha} \PP(\vec{\Theta} \in \cdot),
	\end{equation}
	for every $u>0$ as $x \toi$, and $\Rightarrow$ denotes the weak convergence of measures.  
	{
	By the regular variation of the one-dimensional marginal distributions of the stationary field $\bsX^{\infty}$, there exists a sequence $(a_n)_n$, $a_n\toi$, and a constant {$\rho\in [0,1]$} 
	 and
	Radon measure $\mu$ on $\overline{\RR}\setminus\{0\}$ given by
	\begin{equation*}
	\mu(dy)=\rho\alpha y^{-\alpha-1} \indicator_{(0,\infty)}(y)dy + (1-\rho)\alpha (-y)^{-\alpha-1} \indicator_{(-\infty,0)}(y)dy
	\end{equation*}}
	 such that a random variable $X$
	with the same distribution as the $X_{ij}$'s satisfies
	$$n\PP(X /a_n \in \cdot) \stackrel{v}{\longrightarrow} \mu\,,$$
		where $\stackrel{v}{\longrightarrow}$ denotes vague convergence on $\overline{\RR}\setminus \{0\}$. 
	In particular, $n\PP(|X| > a_n u)\to u^{-\alpha}$ for all $u>0$.
			
	It is convenient to extend $\bsX^{\infty}$ to be a stationary regularly varying random field indexed over 
	the integer lattice $\ZZ^2$.
	By results of \cite{basrak:segers:2009,basrak:planinic:2020}, the regular variation of the stationary field
	$\bsX^{\infty}$ 
	is equivalent to the existence of a tail random field
	denoted by  $\bsY=(Y_{ij})_{i,j\in\ZZ}$, which
	satisfies $\PP(|Y_{00}| > y)=y^{-\alpha}$ for $y \geq 1$ and, as $x\to \infty$,
	\begin{align}
	\label{eq:tailprocess}
	\left ( \{x^{-1}X_{ij}, i,j \in \ZZ\} \, \big| \, |X_{00}| > x \right)
	\fidi \{Y_{ij},i,j \in \ZZ\} \; ,
	\end{align}
	where $\fidi$ denotes convergence of the finite-dimensional distributions. Moreover, the so-called
	spectral tail process $\{\Theta_{ij},i,j \in \ZZ\}$, defined by $\Theta_{ij} := Y_{ij}/|Y_{00}|$, $i,j \in \ZZ$, 
	turns out to be independent of $|Y_{00}|$ and satisfies, as $x\toi$,
	\begin{equation}
	\label{eq:theta}
	\left( \{|X_{00}|^{-1} X_{ij}, {i, j} \in \ZZ\} \, \big| \, |X_{00}| > x \right)
	\fidi \{\Theta_{ij}, i,j \in \ZZ\} \,. 
	\end{equation}

	Consider now  the restriction of such a field to a rectangular area, say of size $n\times n$ for simplicity.
	As in \cite{basrak:planinic:2020}, one can  study the growth of the values in the increasing squares $(X_{ij})_{1\leq i , j\leq n}$.
	In order to obtain a nontrivial asymptotic theory, it is necessary to restrict the dependence in the array $(X_{ij})$.  Therefore, we assume throughout that the array $(X_{ij})_{i,j\in \ZZ}$ is $m$-dependent for some nonnegative integer $m$.  It is known (cf. \cite[Theorem 2.1]{bradley2005basic})  that  this notion of $m$--dependence on the lattice $\ZZ^2$ is actually equivalent to the $\beta$--mixing condition.
%
%
The property of $m$-dependence implies the crucial fact that 	
	\begin{align}\label{eq: 0fdd}
	Y_{ij} = 0 \ \text{almost surely, if} \max\{|i|,|j|\}>m.
	\end{align}	
	Moreover, by the same token, for any $\vep >0$ and indices $(i,j)$
	and $(k,\ell)$ such that $\max\{|i-k|,|j-\ell|\}>m$, we have
	 $\P( |Y_{ij}| > \vep \,, |Y_{k\ell}| > \vep ) =0$. In other words $(Y_{ij})_{i,j}$ has no two nonzero elements with indices separated by more than $m$.

	
	Let $l_0$ be the space of real--valued arrays indexed over $\ZZ^2$ and converging to zero away from the origin, i.e., $$l_0 :=
	\{\bx= ({x}_{ij})_{i,j\in\ZZ} : \lim_{|(i,j)|\to\infty}x_{ij}= 0\}.$$
	As explained above, with probability one, $(Y_{ij})_{i,j}\in  l_{0,m} \subseteq l_0$  where
\begin{align}\label{def:lom}
 l_{0,m} :=  &
\{\bx= ({x}_{ij}) \in l_0  :  \mbox{there exist no two indices $(i,j)$
	and $(k,\ell)$ such that} \\ \nn 
& |x_{ij}| \not =0, |x_{k\ell}| \not = 0 \text{ and }\mbox{$\max\{|i-k|,|j-\ell|\}>m$ }
   \}.
\end{align}
If we endow $l_0$ with  the uniform norm
$	\|\bx\|_\infty = \sup_{i,j\in\ZZ} |x_{ij}| \; ,$
it becomes a separable Banach space.
	
		Define shift operators $\tau,\tau'$ on $l_0$ by $(\tau\bx)_{i,j} = (x_{i+1,j})_{i,j}$ and $(\tau'\bx)_{i,j} = (x_{i,j+1})_{i,j}$. Introduce an
		equivalence relation $\sim$ on $l_0$ by letting $\bx \sim \by $ if $\by=\tau^k\tau'^l\bx$ for some
		$k,l\in\ZZ$.  In the sequel, we consider the quotient space
	$$
		\lo := l_0/\sim\,,\quad 		\tilde{l}_{0,m} := l_{0,m}/\sim\,,
		$$
		and define a distance $\tilde{d}:\lo\times\lo\longrightarrow [0,\infty)$ by
		$$
		\tilde{d}(\tilde{\bx},\tilde{\by}) :=
		\inf\{\|\bx'-\by'\|_\infty:\bx'\in\tilde{\bx},\by'\in\tilde{\by}\}=\inf\{\|\tau^k \tau'^l \bx- \by\|_\infty:k,l\in\ZZ\}\,,
		$$
		for all $\tilde{\bx},\tilde{\by}\in \lo$, and all $\bx\in\tilde{\bx},\by\in\tilde{\by}$.  It follows then, cf. \cite{basrak2016invariance}, that
		$\lo$ is a separable and complete
		metric space with respect to $\tilde{d}$.
		Moreover, one can naturally embed any (matrix) space $
		\RR^{d \times d'}$, ${d,d'\geq1}$ into $\lo$ 
		by concatenating zeros around a given array in $\RR^{d \times d'}$. In particular, any finite block of observations $(X_{ij})_{1\leq i \leq \ell, 1\leq j\leq k}$ for $\ell, k \geq 1$ can be considered an element in $\lo$. 

		Due to $m$-dependence, it follows from \cite{basrak:planinic:2020} that  the following quantity is strictly positive 
		\begin{equation}\label{def:theta}
		\theta:=\PP \left(\sup_{(i,j)<  (0,0) } |Y_{ij}| \leq 1 \right)\,,
		\end{equation}
where we apply the lexicographic order on $\ZZ^2$, i.e., $(i,j) < (i',j')$ if either (a) $i<i'$ or (b) $i=i'$ and $j<j'$.
 Denote by 
	  $\bsZ=({Z}_{ij})_{i,j\in \ZZ}$ an array of random variables  distributed as $\bsY$  
	 conditioned on the event $\{\sup_{(i,j)<  (0,0)}|Y_{ij}| \leq 1\}$.
 That is, the law of $\bsZ$ is given by    
 \begin{equation}
 \label{eq:z}
 \mathcal{L}\Big( \{{Z}_{ij}, i,j\in\ZZ\}\Big) = \mathcal{L}\left( \{{Y}_{ij}, i,j\in\ZZ\}\, \Big|\, \sup_{(i,j)<  (0,0) } |Y_{ij}| \leq 1 \right)\,.
 \end{equation}
 By stationarity, $\theta$ in \eqref{def:theta} is the reciprocal of the expected number of $Z_{ij}$'s with modulus greater or equal to 1, see Remark 3.6 in \cite{basrak:planinic:2020}. 
Moreover, the maximum of $(|X_{ij}|)_{1\leq i,j\leq n}$ asymptotically behaves as the maximum of $ \lfloor \theta \cdot n^2 \rfloor$ \iid\  random variables with the same marginal distribution,  such a constant $\theta$ is often called the extremal index in the literature,  cf.~Remark 3.11 in \cite{basrak:planinic:2020}.
 Note that by considering the conditioned array $({Z}_{ij})_{i,j\in \ZZ}$,  one
 	cancels bias towards blocks with a greater number of high level exceedances
 	inherent in the definition  of the tail array $({Y}_{ij})_{i,j\in \ZZ}$.
	Such conditioning also provides  a common reference (or anchoring)  point for the tail field $\bsY$ by letting  $Z_{00}$ be the `left-most' element greater than 1. Clearly,  $\bsZ$  is also a random element of $l_{0,m} \subseteq l_0$. {Hence,  \eqref{eq:z} immediately induces a distribution for  $\bsZ$ on both $\lom\subseteq\lo$ and  $l_{0,m}\subseteq l_0$ in a natural way.} 
	In particular, the random variable
	\begin{equation}\label{def:L_Z}
	L_Z:=\sup_{i,j\in \ZZ} |Z_{ij}|
	\end{equation} 
	is a.s. finite and larger than 1 since $\PP(|Y_{00}| > 1)=1$. Using the regular
	variation property  one can show (see Section 2 of \cite{basrak2016complete})
	that $\P(L_Z>v)= v^{-\alpha}$ for $v\geq 1$.

	We also define a normalized array $\bsQ=({Q}_{ij})_{i,j\in\ZZ}$  in $\tilde l_{0,m}$ as the equivalence class of
	\begin{equation}
	\label{eq:q}
	Q_{ij} := Z_{ij} / L_Z \; , \ i,j\in \ZZ \; .
	\end{equation}
	{Observe that $\bsQ \in \mathbb{S}$, where $$\mathbb{S}:=\{\tilde{\bx}\in\lo:\|\tilde{\bx}\|_\infty=1\}$$ denotes the unit sphere in $\lo$ in the metric induced by the $\|\cdot\|_\infty$ norm.}
	It turns out  that
	$L_Z$ and $\bsQ$ are independent \cite{basrak:planinic:2020}.
	Consider now a block of observations $(X_{ij})_{1\leq i,j \leq r_n}$, conditioned on the event 
	$\{M_{r_n} > a_{n^2} u\}$ where {$r_n\to\ff$} and $$M_{r_n}:=\max_{1\le \max(i,j)\le r_{n}} |X_{ij}|.$$ 
	After conditioning, normalizing, {and quotienting out by $\sim$}, the law of such a block has a limiting distribution equal to the law of $\bsZ$ as long as $\lim_{n\toi}r_n/n=0$.
The following result is a direct consequence of Proposition 3.8 in \cite{basrak:planinic:2020}:
%
\begin{proposition} \label{prop:block:conv}
		\label{lem:conv-cluster}
		Let	$k_n := \lfloor n/ r_n \rfloor$. Under $m$--dependence and regular variation conditions, for every $u >0$,\\
		i)  
		\[
		  k_n^2 \PP \left( M_{r_n} > a_{n^2} u \right) \to \theta u ^{-\alpha}
		  \,,
		\]		
	ii)
		\begin{align*}
		\law \left( (a_{n^2}u) ^{-1} (X_{ij})_{1\leq i \leq r_n, 1\leq j\leq r_n}  \, \Big|\, M_{r_n} > a_{n^2} u\right)  \Rightarrow \law
		\left(\bsZ \right)\,,
		\end{align*}
		as $n \toi$ in $\lo$. Moreover, the array $\bsQ$ and random variable $L_Z$ introduced in  \eqref{def:L_Z} and \eqref{eq:q} are independent.
\end{proposition}

 Heuristically, a typical square block of observations which has at least one exceedance  above a large threshold, behaves asymptotically as  the conditioned tail field $\bsZ$ {(viewed as a random element of $\lo$)}. 
Due to  $m$-dependence, one can also show that
		\begin{align*}
		\law \left( (a_{n^2} u) ^{-1} (X_{ij})_{-m\leq i, j\leq m}  \, \Big|\,  \max_{ \max(|i|,|j|)\le m, (i,j)<(0,0)} |X_{ij}|\leq a_{n^2} u\,,  |X_{00}|  > a_{n^2}u\right)  \Rightarrow \law
		\left(\bsZ \right)\,,
		\end{align*} 
	cf.	\cite{basrak:planinic:2020}. 
Let
	\begin{align}\label{eq:blockdef}
	\bsB_{kl} :=
	\{ X_{ij}/a_{n^2}  : {(k-1)r_n+1\leq i \leq  kr_n, (l-1)r_n+1 \leq j \leq lr_n} \}
	\end{align}
and denote by $$\|{\bsB_{kl}}\|_{\max} := \max \{ |X_{ij}/a_{n^2}|  : {(k-1)r_n+1\leq i \leq  kr_n, (l-1)r_n+1 \leq j \leq lr_n} \}.$$
Due to stationarity and $m$--dependence, individual blocks are equally distributed 
\footnote{In fact, Proposition \ref{prop:block:conv} shows that  on $\lo$
for every $u \geq 1$, 
{ 
\begin{align}
\PP \left(  \|{\bsB_{11}}\|_{\max} >  u\,, \bsB_{11}/\|{\bsB_{11}}\|_{\max} \in \cdot \,  \, \Big|\, \|{\bsB_{11}}\|_{\max} > 1 \right)  \Rightarrow 
u^{-\alpha}         \PP\left(\bsQ \in \cdot \right)\,,
\end{align}}
cf. \eqref{def:RV_d}, hence
the individual  blocks
can be considered asymptotically regularly varying, although their distribution clearly changes with $n$.} random elements in $\lo$ and only weakly dependent. 

Now consider the point processes 
$$\PPC_n^B := \sum_{k,l=1}^{k_n} \delta_{((k,l)/k_n,\bsB_{kl})}\,.$$ 
	Let $\mathcal{M}^0_p([0,1]^2\times\lo)$ denote the set of point measures on $[0,1]^2\times\lo$ that are finite outside a neighborhood of the set $[0,1]^2 \times \{\bsO\}$ endowed with the appropriate vague topology, cf. \cite{basrak:planinic:2020}, here $\bsO$ represents the sequence of all 0's. The following is an immediate consequence of Theorem 3.9 in \cite{basrak:planinic:2020}:
	
	\begin{proposition}
		\label{thm:PPconvInLo2}
		Let $\bsX^{\infty}$ be a stationary $m$-dependent regularly varying array with tail index $\alpha$ and $(r_n)$ a sequence such that $r_n/n\to 0 $ and $r_n \toi$. Then $\PPC_n^B \dto
		\PPC$ in $\mathcal{M}^0_p([0,1]^2\times\lo)$ where $N$  is a Poisson  process 
		with the following representation
		\begin{align}
		\PPC = \sum_{i=1}^\infty\delta_{(T_i, P_i\bsQ_{i})} \; , \label{eq:Repr}
		\end{align}
		where
		\begin{enumerate}[(i)]
			\item $\sum_{i=1}^\infty\delta_{(T_i,P_i)}$ is a Poisson point process on {$[0,1]^2\times(0,\infty)$}
			with intensity measure $\theta\cdot Leb \times d(-y^{-\alpha})$ {where $\theta$ is as in \eqref{def:theta} {and $P_1>P_2>\cdots$};}
			\item $({\bsQ}_i)_{i\in\N}$ is a sequence of \iid\ elements in $\mathbb{S}$, independent of
			$\sum_{i=1}^\infty\delta_{(T_i, P_i)}$ and with common distribution equal to the distribution of
			$\bsQ$ in \eqref{eq:q}.
		\end{enumerate}
	\end{proposition}
	
\begin{remark}
	The infinite $m$--dependent array 
		$
		(X_{ij})_{i, j}\
		$ 
		does not have to be restricted to the square $(i,j)/n \in [0,1]^2$.
		 {As it is clear from the proof, the theorem above also holds in an arbitrary rectangle $(i,j)/n \in [0,a]\times[0,b],\ a, b>0$.
		Therefore,} the result extends to the
		 point processes 	$ \sum_{k,l\ge 1} \delta_{((k,l)/k_n,\bsB_{kl})}$ and convergence in the space of point measures in $\mathcal{M}^0_p([0,\infty)^2\times\lo)$ {in an appropriate vague topology. Observe simply that the vague convergence of measures in $\mathcal{M}^0_p$ corresponds to the convergence of
		 integrals  $\int f d \mu_n \to \int  f d \mu$ for all bounded, continuous functions with a restricted (or bounded) support, see \cite{basrak:planinic2018note}. If we restrict the support of such functions on $[0,\infty)^2\times\lo$ to  the sets of the form $[0,a]^2 \times \lo$
		 excluding some neighborhood of  $[0,a]^2 \times \{\bsO\}$, over  $a>0$, then this extension becomes immediate.}
		 
\end{remark}	
	
	\section{Results for $m$-dependent random matrices}
	\subsection{Extreme eigenvalues of heavy-tailed $m$-dependent Wigner matrices}\label{sec:refl}

	We will now impose the condition that $\bsX^{\infty}$ is $m$-dependent (see the definition in Section \ref{sec:random fields}).
Recall from \eqref{eq:defwigner} that $\hA= (\h A_{ij}) =(\h X_{ij}/a_{n^2})$. Thus $\hA$ is an $m$-dependent {\bf heavy-tailed Wigner matrix}.

	For the points $(P_i,\bsQ_i)_i$ of \Cref{thm:PPconvInLo2}, 
denote by
	\begin{align}
	\label{eq:sigma^i}
	\sigma_{(i,1)} \geq \sigma_{(i,2)} \geq \sigma_{(i,3)} \geq \dots 
	\end{align}
	the ordered singular values 
	 of $\bsQ_i$ (the $\bsQ_i$ exist and are well-defined by Proposition \ref{prop:block:conv}).
	They are random, but also independent of the points $(P_i)_{i\in\N}$, which form a Poisson point process on {$(0,\infty)$} with intensity measure $d(-\theta y^{-\alpha})$ {such that $P_1>P_2>\cdots$}.

	Our main result characterizes the joint limit of the point processes $N_n^{\pm}$ of eigenvalues of $\hA$, where
\begin{equation}\label{def:Npm}
N_n^+ := \sum_{i=1}^n \delta_{\lambda_i(\hA)} \indicator_{\{\lambda_i(\hA)>0\}}\quad \text{and} \quad 
N_n^- := \sum_{i=1}^n \delta_{\lambda_i(\hA)} \indicator_{\{\lambda_i(\hA)<0\}}\,.
\end{equation}	
	\begin{theorem}
		\label{tm:ReflMat}
Let $\bsX^{\infty}$ be a stationary $m$-dependent regularly varying array with tail index $\alpha\in (0,4)$ and consider the Wigner matrix $\hA$ defined in \eqref{eq:defwigner}.
		 If $2\le \alpha<4$ assume {in addition} that $\EE X_{11}= 0$.
Then we have the joint convergence
\begin{equation}\label{eq:ddeede}
(N_n^+,N_n^-) \cid \Big(\sum_{i=1}^{\infty} \sum_{j=1}^{m+1}  \delta_{P_i \sigma_{(i,j)}}\,, \, \sum_{i=1}^{\infty} \sum_{j=1}^{m+1}  \delta_{-P_i \sigma_{(i,j)}} \Big)\,, \qquad \nto\,,
\end{equation}
where $P_i$ and $\sigma_{(i,j)}$ are as in \eqref{eq:sigma^i}.
The weak convergence of the point processes holds in the space of 
point measures on $(0,\infty)$ and $ (-\infty,0)$ respectively equipped with the vague topology.
	\end{theorem}

If $\bsX^{\infty}$ is an i.i.d. field, we have that $\theta = 1$, $|Q_{00}| =1$ and $Q_{ij} =0$ for $(i,j)\not = (0,0)$ (see \eqref{eq: 0fdd}). Thus we have $\sigma_{(i,1)}=1$ and $\sigma_{(i,j)}=0$ for all $j>1$. By Theorem \ref{tm:ReflMat}, we have $N_n^+ \cid \sum_{i=1}^{\infty} \delta_{P_i}$, that is, we obtain Theorem 1 in \cite{auffinger:arous:peche:2009} as a special case.

The weak convergence of the point processes of the eigenvalues of $\hA$ 
in Theorem~\ref{tm:ReflMat} allows one to use the conventional tools in this
field; see \cite{resnick:2007,resnick:1987}. In case $\bsX^{\infty}$ is an i.i.d. field, an immediate consequence is
\beao
\Big(\max_{i=1, \ldots,n} \lambda_i(\hA), \min_{i=1, \ldots,n} \lambda_i(\hA) \Big) \cid (P_1,-P_1)\,.
\eeao
 More generally, we obtain the following result about the finite dimensional distributions.
\begin{corollary}\label{cor:1}
Let $K\in\N$ and assume the conditions of Theorem \ref{tm:ReflMat}.
Then the $K$ largest eigenvalues of the heavy-tailed Wigner matrix $\hA$ converge in distribution to the $K$ largest points in a point process of the form	
		\begin{align*}
		\sum_{i=1}^\infty \dsum_{j=1}^{m+1} \delta_{P_i \sigma_{(i,j)}}\,.
		\end{align*}
\end{corollary}

Since the eigenvalues of $\hA$ are real, the set of singular values $\{\sigma_i(\hA)\}$ of $\hA$ coincides with $\{|\lambda_i(\hA)|\}$. Therefore we get the following corollary of Theorem \ref{tm:ReflMat}. 
	\begin{corollary}
Under the assumptions of Theorem \ref{tm:ReflMat}, we have
		\begin{align*}
		\sum_{i=1}^n \delta_{\sigma_i(\hA)} \cid \sum_{i=1}^{\infty} \sum_{j=1}^{m+1}  {2}\delta_{P_i \sigma_{(i,j)}}
		\end{align*}
		as $n \toi$, where $P_i$ and $\sigma_{(i,j)}$ are as in \eqref{eq:sigma^i}.
\end{corollary}

\begin{remark}\label{rem:111}
Using Eq. \eqref{eq:eigenvec} in \Cref{lem:factor2} it is also possible to describe the eigenvectors of $\hA$ in terms of the eigenvectors of the matrices $\bsQ_i' \bsQ_i$. It turns out that the eigenvectors associated with the $k$th-largest or smallest eigenvalues of $\hA$ are localized. This property was already observed for the i.i.d.\ case in \cite[Theorem 1.1]{benaych:peche:2014}. 
\end{remark}

\subsubsection{Elements of the proof of Theorem \ref{tm:ReflMat}}

	The proof relies on a classical result about perturbations of the spectrum, which states that for
	two Hermitian $n\times n$ matrices $\bsH$ and $\bsE$, the ordered eigenvalues $(\lambda_{j})_{1\le j\le n}$ of the matrices $ \bsH$ and  $\bsH+\bsE$ satisfy 
	Weyl's inequality
	\begin{align}\label{eq:weyl}
	\max_{j=1,\ldots,n} | \lambda_{j}  (\bsH)   -  \lambda_{j}  (\bsH+\bsE )   |
	\leq \| \bsE \|\,,
	\end{align}
	where $\norm{\cdot}$ denotes the spectral norm. It is well known that this norm is further bounded by the Frobenius norm, i.e.,
	\begin{align}\label{eq:frobenius}
	\| \bsE \|^2
	\leq \| \bsE \|_F^2 :=  \dsum_{i,j} E_{ij}^2 \,.
	\end{align}

As in \cite{auffinger:arous:peche:2009}, our strategy is to truncate the matrix entries by removing all small-enough entries, and then to use Weyl's inequality to show that our truncation is insignificant in the scaling limit. The main difference is that we do this in a block matrix setting, and so in particular, we actually remove all {\it blocks} that are small enough in the normed space $l_0$.  We then use the results of {Section \ref{sec:random fields}} to show that the eigenvalues formed from only the significant blocks converge to a Poisson cluster process. 
	
	We now describe the block matrices under consideration. Afterwards, we will describe our truncation and show that it does not affect the limiting eigenvalues, i.e., that eigenvalues associated to separate blocks have only `weak interactions'.
	
	As in \eqref{eq:blockdef} above, we can group  the  entries of $\hA$ into blocks of size $r_n\times r_n$ and 
	set again,
	$k_n = \lfloor n/ r_n \rfloor$.  {In fact, it will be evident from the proof that there is no loss of generality by letting $n$ be such that $k_n = n/ r_n$, which we will henceforth assume.}  The $kl$ block is denoted
	\[
	\hB_{kl} =\hB_{n,kl} :=
	\big( {\hat A}_{ij}  : i\in ((k-1)r_n, kr_n], j\in((l-1)r_n, lr_n] \big)\,,
	\] 
	and the array of blocks $(\hB_{kl})$ form the block matrix 
	\begin{align} \label{eq:block form of reflected matrix}
	\hA=  
	\begin{pmatrix}
	\hB_{11} & \hB_{12}  & \cdots & \hB_{1k_n}\\
	\vdots & \vdots & \ddots & \vdots\\
	\hB_{k_n 1}  &   \hB_{k_n 2}      &\ldots & \hB_{k_nk_n}
	\end{pmatrix}.
	\end{align}
	In particular, $\hB_{kl}$ (which is the $kl$-entry of the $k_n\times k_n$ block form of the matrix $\hA$) is itself an $r_n\times r_n$ matrix, which is indicated by the boldface type, however, only the diagonal blocks $\hB_{kk}$ are generally Hermitian. 
	
		{ We want to utilize \Cref{thm:PPconvInLo2}.   For the square $[0,1]^2$ in the $xy$-plane in \Cref{thm:PPconvInLo2},  we reverse the orientation of the $y$-axis (to go downward) in order to match with the natural numbering of rows in a matrix.
	Now let the first coordinate of the ordered pair $(\h T_i,P_i\bsQ_i)$  be a point in the triangle lying below the line $y=x$ inside $[0,1]^2$ (under the reversed orientation).}  
	This triangle corresponds to a rescaled limit of positions $(k,l)$ with $k<l$, as $n\to\infty$, in the  upper triangle of a sequence of square matrices:
	\begin{align}\label{eq:conv to ppp}
	 \sum_{k=1}^{k_n}\sum_{l=k+1}^{k_n} \delta_{((k,l)/k_n,\hB_{kl})}  
	\dto
	 \sum_{i=1}^\infty\delta_{(\h T_i, P_i\bsQ_{i})} \; , \quad \nto\,,
	\end{align}
{	in the space of point measures on the state space $\Delta \times\lo$, where
	$\Delta=  \{(x,y): x<y\}$}.
	One can easily see that the contribution of the diagonal blocks (which have a different distribution than off-diagonal blocks since they are Hermitian)
	is asymptotically negligible.

\subsection{Extreme eigenvalues of regularly varying sample covariance matrices}\label{sec:cov}

In this section we consider the spectrum of high-dimensional heavy-tailed $m$-dependent sample covariance matrices constructed from a stationary $m$-dependent field $\bsX^{\infty}$. We start by recalling their definition in \eqref{eq:defdata}. For a sequence of integers $p=p_n$ such that $p_n/n \to \gamma \in (0,\infty) $ we consider the $m$-dependent matrix 
\begin{equation*}
\bsA = \bsA_n  = (X_{ij}/a_{np})_{1\leq i \leq p, 1\leq j\leq n},
\end{equation*}	
and study the spectrum of the Hermitian  
$p\times p$ \textbf{ sample covariance matrix} 
$
\bsA \bsA'\,.
$

\begin{theorem}\label{tm:CovMat}
Let $\bsX^{\infty}$ be a stationary $m$-dependent regularly varying array with tail index $\alpha\in (0,4)$ and consider the data matrix $\bsA$ defined in \eqref{eq:defdata}.
		 If $2\le \alpha<4$ assume {in addition} that $\EE X_{11}= 0$.
Then we have the point process convergence
\begin{equation}\label{eq:ddeede1}
\sum_{i=1}^p \delta_{\sigma_i(\bsA)} \cid  \sum_{i=1}^{\infty} \sum_{j=1}^{m+1}  \delta_{P_i \sigma_{(i,j)}} \,, \qquad \nto\,,
\end{equation}
where $P_i$ and $\sigma_{(i,j)}$ are as in \eqref{eq:sigma^i}.
The weak convergence of the point processes holds in the space of 
point measures with state space $(0,\infty)$ equipped with the vague topology.
	\end{theorem}
Theorem \ref{tm:CovMat} can be reformulated  for the eigenvalues of the sample covariance matrices $\bsA \bsA'$. Using $\lambda_i(\bsA \bsA')= \sigma_i^2(\bsA)$ and the continuous mapping theorem, we get that 
\begin{equation}\label{eq:ddeede2}
\sum_{i=1}^p \delta_{\lambda_i(\bsA\bsA')} \cid  \sum_{i=1}^{\infty} \sum_{j=1}^{m+1}  \delta_{P_i^2 \sigma_{(i,j)}^2} \,, \qquad \nto\,\,.
\end{equation}
If $\bsX^{\infty}$ is an i.i.d. field, equation \eqref{eq:ddeede2} reads as
\begin{equation*}
\sum_{i=1}^p \delta_{\lambda_i(\bsA \bsA')} \cid  \sum_{i=1}^{\infty}  \delta_{P_i^2} \,, \qquad \nto\,\,,
\end{equation*}
with $\theta=1$ in the definition of the points $(P_i)$.
Thus, Theorem \ref{tm:CovMat} generalizes Theorem 2 in \cite{auffinger:arous:peche:2009}.

Similarly to Corollary \ref{cor:1}, one can derive the joint convergence of the $K$ largest eigenvalues of $\bsA\bsA'$ from Theorem \ref{tm:CovMat}.
\begin{corollary}\label{cor:dfg}
Let $K\in\N$ and assume the conditions of Theorem \ref{tm:CovMat}.
Then the $K$ largest eigenvalues of the heavy-tailed sample covariance matrix $\bsA\bsA'$ converge in distribution to the $K$ largest points in a point process of the form	
		\begin{align*}
		\sum_{i=1}^{\infty} \sum_{j=1}^{m+1}  \delta_{P_i^2 \sigma_{(i,j)}^2}\,.
		\end{align*}
\end{corollary}

\begin{remark}\label{rem:211}
As in Remark \ref{rem:111}, we can describe the eigenvectors of $\bsA\bsA'$ in terms of the eigenvectors of the matrices $\bsQ_i' \bsQ_i$. More precisely, by mimicking the arguments in the proof of \cite[Theorem 3.11]{heiny:mikosch:2017:iid} or \cite[Theorem 3.7]{heiny:mikosch:2019:stochvol} one can show that the eigenvector associated with the $k$th-largest eigenvalue of $\bsA\bsA'$ is an appropriately shifted version of an eigenvector of some $\bsQ_i' \bsQ_i$. Since under $m$-dependence $\bsQ_i' \bsQ_i$ are zero outside of some block of size $(m+1)\times (m+1)$, this implies that the eigenvectors of $\bsA\bsA'$ are localized asymptotically. 
\end{remark}

	\subsection{Examples of $m$-dependent matrix ensembles}\label{sec:example} 
	
%

\noindent{\bf Two-dimensional linear processes}

Consider a two-dimensional moving average structure of order $m>0$: 
\begin{equation}\label{eq:1}
X_{it}=\sum_{k,l=0}^{m} h_{kl} Z_{i-k,t-l}\,,\qquad i,t\in\Z\,,
\end{equation}
where $(Z_{it})_{i,t\in \Z}$ is a field of {i.i.d.} regularly varying random variables with index $\alpha\in (0,4)$ and $\bfH=(h_{kl})_{k,l=0,\ldots,m}$ is an array of real numbers. If $\alpha\ge2$, additionally assume that $\E[Z_{11}]=0$.
From Example 3.1 in \cite{basrak:planinic:2020} we know that 
\begin{equation*}
\bsQ \eid \left( \frac{K h_{ij}}{\max_{k,l}|h_{kl}|} \right)_{ij} \quad \mbox{ and } \quad \theta= \frac{\max_{k,l} |h_{kl}|^{\alpha}}{\sum_{k,l} |h_{kl}|^{\alpha}}\,,
\end{equation*}
where  $K$ is  $\pm 1$-valued random variable, such that 
	$\PP (K=1) = \lim_{x\toi} \PP (Z_{11} >x) / \PP(|Z_{11}|>x)$.
If we denote the ordered eigenvalues of $\bfH\bfH'$ by $v_1\ge \cdots \ge v_{m+1}$, then the (possible) non-zero singular values of $\bsQ$ are given by
\begin{equation*}
\sigma_j(\bsQ)=\frac{v_j^{1/2}}{\max_{k,l} |h_{kl}|}\,, \qquad j=1,\ldots,m+1\,.
\end{equation*}
For the Wigner matrix $\hA$ defined in \eqref{eq:defwigner}, Theorem \ref{tm:ReflMat} yields
\begin{equation}\label{eq:wign1}
N_n^+ \cid \sum_{i=1}^{\infty} \sum_{j=1}^{m+1}  \delta_{P_i v_j^{1/2}/\max_{k,l} |h_{kl}|}\,, \qquad \nto\,,
\end{equation}
where the $(P_i)_{i\in\N}$ form a Poisson point process on {$(0,\infty)$} with intensity measure $d(-\theta y^{-\alpha})$ and $P_1>P_2>\cdots$.

For the sample covariance matrices $\bsA\bsA'$, equation \eqref{eq:ddeede2} gives
\begin{equation}\label{eq:scov1}
\sum_{i=1}^p \delta_{\lambda_i(\bsA\bsA')} \cid  \sum_{i=1}^{\infty} \sum_{j=1}^{m+1}  \delta_{P_i^2 v_j/\max_{k,l} |h_{kl}|^2} \,, \qquad \nto\,\,.
\end{equation}
In the special case,	$X_{it}= Z_{it}+ Z_{i,t-1}-2 (Z_{i-1,t}- Z_{i-1,t-1}),i,t\in\Z$, we have
	$v_1=8$ and $v_2=2$.
	Using Corollary \ref{cor:dfg} we find the joint limit of the two largest eigenvalues of the sample covariance matrix:
	\begin{equation*}
	\big(\lambda_1(\bsA\bsA'),\lambda_2(\bsA\bsA')\big) \cid
	\big(2 P_1^2, \tfrac{P_1^2}{2} \vee 2 P_2^2 \big)\,, \qquad \nto\,.
	\end{equation*}



\noindent{\bf Two-dimensional max--linear processes}

	Instead of \eqref{eq:1} one can consider moving maxima 
	\begin{equation*}
	X_{it}=\bigvee_{k,l=0}^m h_{kl} Z_{i-k,t-l}\,,\qquad i,t\in\Z\,,
	\end{equation*}
	with  nonnegative coefficients $h_{kl}$ and nonnegative regularly varying i.i.d.~noise
	$(Z_{i,t})$ {with index $\alpha\in (0,2)$}. It is straightforward to see
	that one again ends up with an $m$--dependent regularly varying array which has the
	same parameter $\theta$ and the same distribution of $\bsQ$ as in the moving average process above. Therefore the limiting relations in \eqref{eq:wign1} and \eqref{eq:scov1} hold unaltered for this process as well, we refer to
	Deheuvels~\cite{deheuvels:1983} for an introduction to moving maxima model
	in  such a model in a time series context.
	
	\noindent{\bf Two-dimensional linear processes, random coefficients}

	An interesting way of generalizing \eqref{eq:1} is to consider a moving average field with stationary and random coefficients $h_{k,l}$ { independent of $(Z_{it})$}. The analysis of such a field is in general more technical but can been done, for a discussion of the corresponding model in a time series context and related references see the recent book by Kulik and Soulier 
	\cite{kulik:soulier:2020}.
	Here, for simplicity consider 
	\begin{equation*}
	X_{it}= 4 Z_{it}+ \vep_{i-1,t} Z_{i-1,t}+ 3 Z_{i-1,t-1}\,,\quad i,t\in\Z\,,
	\end{equation*}
	with $(Z_{it})_{i,t\in \Z}$ as above and independent of an i.i.d. sequence
	$(\vep_{it})_{i,t\in \Z}$ consisting of  Bernoulli random variables with parameter $q \in (0,1)$.
	It can be shown by direct calculation, that the sequence $(X_{it})$ is $1$--dependent, stationary and regularly varying. In this case 
	\begin{equation*}
	\bsQ \eid \left( 
	\begin{array}{cc}
	1 & 0 \\
	\vep_{11}/4 & 3/4 \\
	\end{array} \right) \quad \mbox{ and } \quad \theta =\frac{4^{\alpha}}{4^\alpha+q+3^\alpha}\,.
	\end{equation*}
	Note that the non-zero singular values of $\bsQ$ are $(18/16,8/16)$ with probability $q$
	and $(1,9/16)$  with probability $1-q$. Thus, in this case for
	the sample covariance matrices $\bsA\bsA'$, equation \eqref{eq:ddeede2} gives
	\begin{equation*}
	\sum_{i=1}^p \delta_{\lambda_i(\bsA\bsA')} \cid  \sum_{i=1}^{\infty} 
	\left( \delta_{P_i^2 \frac{16+ 2 \vep_i }{16 } } +  \delta_{P_i^2 \frac{8+ 1-\vep_i }{16 }}  \right) \,, \qquad \nto\,\,,
	\end{equation*} 
	where $(\vep_i)$ is an i.i.d.\ Bernoulli sequence with parameter $q \in (0,1)$
	independent of the Poisson process $ \sum_{i=1}^{\infty} 
	\delta_{P_i}$.


Another example with random coefficients is
	$$X_{it}=\vep_{it}\sum_{j,s=0}^m  Z_{i+j,t+s}$$
		with $(Z_{it})_{i,t\in \Z}$ as above and independent of an i.i.d.\ sequence
	$(\vep_{it})_{i,t\in \Z}$ consisting of   Rademacher random variables (mean-zero, $\{-1,1\}$-valued).
	The distribution of
	$\bsQ$ can be viewed as an $(m+1)\times (m+1)$ matrix with { independent} Rademacher entries.
	If  the eigenvalues of $\bfQ\bfQ'$ are equal in distribution to $(V_1, \cdots, V_{m+1})$, and $(V_1^{(i)},\cdots,V_{m+1}^{(i)})_{i\in\N}$ are i.i.d. copies of this random vector, then \eqref{eq:ddeede2} describes the distribution limiting eigenvalues of the sample covariance matrices $\bsA\bsA'$ as
	\begin{equation*}
	\sum_{i=1}^p \delta_{\lambda_i(\bsA\bsA')} \cid  \sum_{i=1}^{\infty} \sum_{j=1}^{m+1}  \delta_{P_i^2 V_j^{(i)}} \,, \qquad \nto\,\,.
	\end{equation*}
	

	\begin{remark}
	In this final example, suppose one takes a family of such sample covariance matrix sequences, one matrix sequence for each $m\in\N$, and normalizes the Rademacher entries of $\bfQ$ by $m^{-1/2}$. 
	If one takes the double limit, first as $n\to\infty$ and then as $m\to\infty$, then by standard random matrix results \cite{soshnikov2002note} and the fact that the properly normalized largest eigenvalue of  $\bfQ\bfQ'/m$ 
 asymptotically  follows a Tracy-Widom(1) distribution, one can obtain the second order fluctuations of the maximal  eigenvalues. In particular, after taking the limit in $n$ for each matrix sequence, for large values of $m$  one will see associated to each $P_i$, a `local' maximal eigenvalue asymptotically of the form $P_i^2(4+2^{4/3}W_1/m^{2/3})$, where $W_1$ follows a Tracy-Widom(1) distribution and is independent of $P_{i}$. (Note that, when applying \cite{soshnikov2002note},  $\gamma=1$ since we have $(m+1)\times (m+1)$ matrices.)
	\end{remark}

\section{Proof of Theorem \ref{tm:ReflMat}}	\label{sec:proofwigner}
	
	{It will be useful in the sequel} to truncate the matrices.
	For the matrix $\bsA=(A_{ij})$ 
	and a constant $\vep >0$, we
	introduce the truncated matrix $\bsA^{>\vep}$ with entries
	\begin{align}\label{eq:truncate}
	A^{>\vep }_{ij} := A_{ij}  \1{{|A_{ij}| > \vep } }\,.
	\end{align}
	Similarly, by $\bsA^{<\vep}$ we denote the remainder $ \bsA^{<\vep }:=
	\bsA - \bsA^{>\vep }$.

Also, for the proof it will be notationally convenient to set
$$b_n:=a_{n^2}.$$
	
	We will show that due to Weyl's inequality, the effect of thresholding by $\vep$ on the eigenvalues is asymptotically negligible as $n\toi$. It is pedagogical to treat separately the case where $\alpha<2$, since the basic structure of the proof will be seen here without having to go into too many details. 
	
	\vspace{2mm}
\subsection{Case: $0<\alpha<2$}
	
	Weyl's inequality yields
	\begin{align}\label{eq:weyl2}
	\max_i | \lambda_{i}  (\hA_n)   -  \lambda_{i}  (\hA_n^{>\vep} )  |
	\leq \| \hA_n^{<\vep} \|\,.
	\end{align}
	If we show that for any 
	$\delta >0$,
	\begin{align} \label{eq:TruncError}
	\lim_{\vep\to 0} \limsup_{n\toi} \PP\left( \|\hA_n^{<\vep} \| > \delta\right) 
	&=0\, ,
	\end{align}
	then it suffices to work with $(\hA_n^{>\vep})$ since the distribution of the point process of its eigenvalues has the same asymptotic behavior as that of the point process of eigenvalues of $(\hA_n)$.

First, we bound the spectral norm by the Frobenius norm and apply Markov's inequality to get
	\begin{align}\label{eq:frob+markov}
		\lim_{\vep\to 0} \limsup_{n\toi} \PP\left( \|\hA_n^{<\vep} \| > \delta\right) 
	& \le  
	\lim_{\vep\to 0} \limsup_{n\toi}  \frac{n^2}{\delta^2 b_n^2}
	\EE\left(   X_{11}^2 \1{|X_{11}| <\vep b_n}  \right) \,.
	\end{align}
	Since the random variable $ X_{11}$ is regularly varying with index $\alpha$,
	for $\alpha < 2$, Karamata's theorem for truncated moments (see \cite{bingham1989regular} or \cite[Appendix B.4]{buraczewski2016stochastic}) yields that the \rhs~in \eqref{eq:frob+markov} behaves as
	\begin{align}\label{eq:e to 0}
	\lim_{\vep\to 0} \lim_{n\toi}  \frac{n^2}{\delta^2 b_n^2}
	b_n^2 \vep^2 \PP\left( {|X_{11}|} >\vep b_n  \right) \frac{\alpha}{2-\alpha}
	& = \lim_{\vep\to 0} \frac{\vep^{2-\alpha}}{\delta^2}  \frac{\alpha}{2-\alpha}=0\,.
	\end{align} 
\subsubsection{Proof of Theorem \ref{tm:ReflMat} for $0<\alpha<2$}	
{In the case of {i.i.d.} entries in the upper triangle of $\hX$,} Lemma 1(c)  of \cite{soshnikov:2004} makes simple use of \eqref{regvar} to show that for any $\vep$, there is at most one nonzero entry in any given row or column of $\hA_n^{>\vep}$  with probability going to 1 as $n\to\infty$, {i.e.},
\begin{align}\label{eq:soshn}
\PP(\exists 1 \leq i \leq n, \,  \exists j\neq k\, \text{ such that }\,   |{\hat A}_{ij}|>\vep, 
\,  \,|{\hat A}_{ik}|>\vep )\to 0.
\end{align}	
Another straightforward observation is that diagonal elements are asymptotically insignificant (c.f. Lemma 1(a) of \cite{soshnikov:2004})
\begin{align*}
\PP(\exists 1 \leq i \leq n \, \text{ such that }\,   |{\hat A}_{ii}|>\vep)\to 0.
\end{align*}
	By the above two facts and the symmetry of the matrix, one can directly check that with probability going to 1, the ordered largest eigenvalues of $\hA_n^{>\vep}$ are the ordered {largest-in-absolute-value  entries of $\hA_n^{>\vep}$  in the upper triangle (${\hat A}_{ij}$ such that $i\le j$), which after taking absolute values, form a
	Poisson process 
	on $(\vep,\infty)$
	with intensity measure $ \frac 1 2  d(y^{-\alpha})$. Similarly, the smallest eigenvalues are the negatives of
	the ordered largest-in-absolute-value  entries of $\hA_n^{>\vep}$ in the upper triangle (see \cite{soshnikov:2004} or Lemma \ref{lem:factor2} for details).}


	Now, applying Weyl's inequality and using \eqref{eq:e to 0},
	reproduces the limiting point process of largest eigenvalues for the sequence $(\hA_n)_n$, as discussed in \cite{soshnikov:2004} (actually, here we gave a slightly different argument than \cite{soshnikov:2004} since he does not use Weyl's inequality).

{
For the situation with $m$-dependence, we will use the following lemma to see that, with respect to the largest eigenvalues, the dependence remains local.}


  
\begin{lemma}\label{lem:En} Let $\alpha\in (0,2)$. For $\vep>0$, consider the block form of $\hA_n$ given in \eqref{eq:block form of reflected matrix}.
	 Then the probability of the event that  $\hA_n^{>\vep}$ has more than one nonzero block (i.e., there is some nonzero entry in the block) in some row or column tends to zero as $n\toi$, i.e., $\lim_{\nto} \P(S_1^{n,\vep})=1$, where $S_1^{n,\vep}$ is the complement of the set
\begin{equation*}
\Big\{\exists 1\le i,j,k\le k_n  \, \text{ with } j\neq k \text{ such that }\,   \|\hB_{ij} \|_{\max}>\vep, \,  \,\|\hB_{ik} \|_{\max}>\vep   \Big\}\, .
\end{equation*}
We also have $\lim_{\nto} \P(S_2^{n,\vep})=1$, where $S_2^{n,\vep}$ is the complement of the set
\begin{equation*}
\Big\{\exists 1\le i\le k_n  \, \text{ such that }\,   \|\hB_{ii} \|_{\max}>\vep \Big\}\, .
\end{equation*}
\end{lemma}
\begin{proof}
	Consider the first row  and blocks $\hB_{1j}$ and $\hB_{1k}$ 
	assuming without loss of generality that $k>j$. {Since $m$ is fixed }
	and $r_n$ goes to infinity (recall $r_n=n/k_n$),  for fixed $j+1<k$ and $n$ large enough, the two blocks are independent from each other and
	\begin{align} \label{eq:independent spectral norms}
	\PP(\|\hB_{1j}\|_{\max}>\vep, 
	\,  \,\|\hB_{1k}\|_{\max}>\vep ) = \PP(M_{r_n}> b_n \vep)^2
	= O(k_n^{-4})\,, 
	\end{align}
by Proposition~\ref{prop:block:conv}\,. If however $k=j+1$, then
\begin{align*}
\lefteqn{\PP\(\|\hB_{1j}\|_{\max}>\vep, 
\,  \,\|\hB_{1k}\|_{\max}>\vep \)} \\
& \leq  O(k_n^{-4}) + \PP\(
	\max \{  \hat X_{i\ell}/b_n  : i\in (0, r_n], \ell \in(j r_n -m, j r_n] \} > \vep \) \\
&= O(k_n^{-4}) + {O ( n^{-1} k_n^{-1})\,.} 
\end{align*}
By stationarity the same upper bound holds for any other row (or any column by symmetry of the matrix $\hA_n^{>\vep}$) and therefore we can use a basic union bound to get
\begin{align*}
 \PP((S_1^{n,\vep})^c) & 
 \le k _n ^3   O(k_n^{-4}) + k _n ^2   {O ( n^{-1} k_n^{-1})}  = O(1/k_n) + O(1/r_n)
 \end{align*}
which tends to 0 since $k_n \toi$ and $r_n \toi$. The proof of $\lim_{\nto} \P(S_2^{n,\vep})=1$ is analogous.
\end{proof}
The above lemma implies that $\P(S^{n,\vep})\to 1$, where $S^{n,\vep}:=S_1^{n,\vep} \cap S_2^{n,\vep}$.


\begin{lemma}\label{lem:factor2}
Let $j\le \text{rank}(\hA^{>\vep})/2$. On the set $S^{n,\vep}$, the $j$th largest and $j$th smallest eigenvalues of the matrix $\hA^{>\vep}$ are given by $\lambda_j$ and $-\lambda_j$, respectively, where $\lambda_j$ is the $j$th largest value (counted with multiplicity) in the set
 $$\bigcup_{(k,l): \|\hB_{kl}\|_{\max}>\vep, k<l} \{\sigma_1(\hB_{kl}^{>\vep}), \ldots, \sigma_{r_n}(\hB_{kl}^{>\vep})\}\,,$$
where $\sigma_i(\hB_{kl}^{>\vep})$ denotes the $i$th largest singular value of $\hB_{kl}^{>\vep}$. 
\end{lemma}
\begin{proof}
We start with some useful facts about the eigenvalues of blockdiagonal matrices.
By the Schur complement formula
$$\det \begin{pmatrix}
A &B\\
C & D
\end{pmatrix}
= \det(A-B D^{-1} C) \det(D)\,,
$$
we see that the following two statements are equivalent for any real valued-matrix $B$:
\begin{itemize}
\item[(i)] $\lambda^2$ is an eigenvalue of $B'B$.
\item[(ii)] $\pm \lambda$ are eigenvalues of $\begin{pmatrix}
0 &B\\
B' & 0
\end{pmatrix}$.
\end{itemize}
Assume $\lambda^2>0$ is an eigenvalue of $B'B$ with associated eigenvector $w$. It is easy to check that 
\begin{equation}\label{eq:dgsrdd}
\begin{pmatrix}
0 &B\\
B' & 0
\end{pmatrix}
\begin{pmatrix}
\lambda^{-1}Bw\\
w
\end{pmatrix} 
= \lambda
\begin{pmatrix}
\lambda^{-1}Bw\\
w
\end{pmatrix}
\quad \!\! \text{and} \quad \!\!
\begin{pmatrix}
0 &B\\
B' & 0
\end{pmatrix}
\begin{pmatrix}
\lambda^{-1}Bw\\
-w
\end{pmatrix}
=-\lambda
\begin{pmatrix}
\lambda^{-1}Bw\\
-w
\end{pmatrix}.
\end{equation}

To see that the lemma holds in the more general setting of the form of matrices satisfied by $\hA^{>\vep}$,  consider $(k,l)$ such that $\|\hB_{kl}\|_{\max}>\vep$ and $k<l$. Let $w_{kl,i}\in \R^{r_n}$ be an eigenvector of ${\hB_{kl}^{>\vep}}' \hB_{kl}^{>\vep}$ associated with eigenvalue $\sigma_i^2(\hB_{kl}^{>\vep})>0$. Recalling that on the set $S^{n,\vep}$, we have that $\hA^{>\vep}$ has at most one nonzero block in every row or column and none on the diagonal, we see analogously to \eqref{eq:dgsrdd} that 
\begin{equation}\label{eq:eigenvec}
\begin{split}
&(0_{(k-1)r_n}, \sigma_i^{-1}(\hB_{kl}^{>\vep})  (\hB_{kl}^{>\vep} w_{kl,i})', 0_{(l-k-1)r_n}, w_{kl,i}',0_{n-l r_n})' \quad \text{ and }\\
& (0_{(k-1)r_n}, \sigma_i^{-1}(\hB_{kl}^{>\vep})  (\hB_{kl}^{>\vep} w_{kl,i})', 0_{(l-k-1)r_n}, -w_{kl,i}',0_{n-l r_n})'
\end{split}
\end{equation}
are eigenvectors of $\hA^{>\vep}$ associated with eigenvalues $\sigma_i(\hB_{kl}^{>\vep})$ and $-\sigma_i(\hB_{kl}^{>\vep})$, respectively. Here, $0_k$ denotes the $k$-dimensional vector of zeros. Here, we have constructed eigenvectors to all nonzero eigenvalues of $\hA^{>\vep}$. 
\end{proof}

\begin{remark}\label{rem:beforeproof}
From now on when we write $N_1=N_2$ for two point processes $N_1,N_2$ we mean that $N_1(D)=N_2(D)$ for any set $D\subset (-\infty,\infty)\backslash \{0\}$. Since $0\notin D$ we then have, for example, 
\begin{equation*}
N_n^+ := \sum_{i=1}^n \delta_{\lambda_i(\hA)} \indicator_{\{\lambda_i(\hA)>0\}}= \sum_{i=1}^n \delta_{\lambda_i(\hA)} \indicator_{\{\lambda_i(\hA)\ge 0\}}\,.
\end{equation*}
\end{remark}

On the set $S^{n,\vep}$, we have by Lemma \ref{lem:factor2} that
\begin{equation}\label{eq:ssddd}
\begin{split}
\Big( &\sum_{i=1}^n \delta_{\lambda_i(\hA^{>\vep})} \indicator_{\{\lambda_i(\hA^{>\vep})\ge 0\}},\,\sum_{i=1}^n \delta_{\lambda_i(\hA^{>\vep})} \indicator_{\{\lambda_i(\hA^{>\vep})\le0\}}\Big)\\
&= \Big( \sum_{k=1}^{k_n}\sum_{l=k+1}^{k_n} \sum_{j=1}^{r_n} \delta_{\sigma_j(\hB_{kl}^{>\vep})}, \sum_{k=1}^{k_n}\sum_{l=k+1}^{k_n} \sum_{j=1}^{r_n} \delta_{-\sigma_j(\hB_{kl}^{>\vep})}   \Big)\,.
\end{split}
\end{equation}
Note that by Lemma \ref{lem:En}, $\P(S^{n,\vep})\to 1$. Therefore it suffices to focus on the positive eigenvalues of $\hA^{>\vep}$. 

Next, we will show that {the singular values of the $\hB_{kl}^{>\vep}$ converge to the singular values of the $(P_i\bsQ_{i})^{>\vep}$ in the right sense.} 

 To see this, let us first observe that the mapping $\bx \mapsto \bx^{>\vep}$ is continuous on $\lo$, except maybe at the exceptional points $\bx =(x_{ij})$ with the property that 
$|x_{ij}| = \vep$ for some ${i,j\in\ZZ}$. However, since the limiting process in Proposition~\ref{thm:PPconvInLo2}  almost surely has no such exceptional points, from \eqref{eq:conv to ppp} we conclude
that, {as a point process in $\lo\backslash \{{\bf 0}\}$, where ${\bf 0}$ is the zero vector,}
\begin{align}
	 \sum_{k=1}^{k_n}\sum_{l=k+1}^{k_n} \delta_{\hB_{kl}^{>\vep}}  
	\dto \sum_{i=1}^\infty\delta_{ (P_i\bsQ_{i}) ^{>\vep}} \; . 
\end{align}
 In other words, {the distribution of the points $(\hB_{kl}^{>\vep})_{k,l}$ converges in $\lo$ to the distribution of} points $(P_i\bsQ_{i})^{>\vep}_i$.

Note that convergence in $\lo$ by itself does not imply {directly that the singular values of $\hB_{kl}^{>\vep}$ converge in distribution} to the singular values of $(P_i\bsQ_{i})^{>\vep}$, as $n\to\infty$, as a point process in $\RR$.
{ The problem is that the elements of $\lo$ are  infinite-dimensional matrices in general. However, if $\bx_n^{>\vep} \to \bx^{>\vep}$ as $n \toi$ in $\lo$ and $\bx \in \tilde{l}_{0,m}$, then 
necessarily $\bx_n^{>\vep} \in \tilde{l}_{0,m}$, for all large enough $n$.} 
%
Taking singular values of elements in $\tilde{l}_{0,m}$ corresponds to taking
singular values in the space of $(m+1)\times (m+1)$ matrices, which on this space is a continuous mapping. 
Therefore
\begin{align}\label{eq:eval conv}
\sum_{k=1}^{k_n}\sum_{l=k+1}^{k_n} \sum_{j=1}^{r_n} \delta_{\sigma_j(\hB_{kl}^{>\vep})}  
\dto \sum_{i=1}^\infty \sum_{j=1}^{m+1} \delta_{ \sigma_j((P_i\bsQ_{i}) ^{>\vep})} \; .
\end{align}
Moreover, {with probability one, as $\vep \to 0$,}
\begin{align*}
\sum_{i=1}^\infty \sum_{j=1}^{m+1} \delta_{ \sigma_j((P_i\bsQ_{i}) ^{>\vep})} \;  
 \to \sum_{i=1}^\infty \sum_{j=1}^{m+1} \delta_{ \sigma_j(P_i\bsQ_{i})}=\sum_{i=1}^{\infty} \sum_{j=1}^{m+1}  \delta_{P_i \sigma_{(i,j)}}\;. 
\end{align*}
Finally, combined with \eqref{eq:weyl2}, \eqref{eq:frob+markov}, and \eqref{eq:e to 0} this gives us that 
\begin{equation*}
N_n^+ \cid \sum_{i=1}^{\infty} \sum_{j=1}^{m+1}  \delta_{P_i \sigma_{(i,j)}}\,, \qquad \nto\,.
\end{equation*}
The joint convergence of the vector $(N_n^+,N_n^-)$ in \eqref{eq:ddeede} follows in view of \eqref{eq:ssddd}.

\subsection{ Case: $2\le \alpha<4$}
	

When $2\le \alpha<4$, the main reason the  analysis is more involved is because we require a truncation level $\vep_n\to 0$ that depends on $n$, which in our case we will set to  $\vep_n:=n^{\beta}/b_n$, where $\beta$ satisfies
{\begin{align}\label{eq:betabounds}
	\frac{4}{3\alpha}<\beta<\frac{2(8-\alpha)}{\alpha(10-\alpha)}.
\end{align}	
The lower bound for $\beta$ is used in order to show that only finitely many blocks remain in any given row of $\hA^{>\vep_n }$ (see Subsection \ref{sec:second trunc} and Lemma \ref{lem:consecblock}), while the upper bound is needed for \eqref{eq:2^2sn} which will ultimately allow us to obtain $\| \hA^{<\vep_n } \| \cip 0$ as $\nto$.
	
	A second complication, but not as significant, is that after truncating all small entries, one may asymptotically have multiple nonzero blocks in a given row or column (see Subsection \ref{sec:second trunc}). This issue will be taken care of by introducing a {second truncation level later on.}

{We will} show that $\| \hA^{< \vep_n } \| \cip 0$. To this end, for any
$\delta >0$, Markov's inequality gives 
\begin{align}\label{eq:frobenius2}
\PP\left( \left\| \hA^{<\vep_n } \right\| > \delta\right) 
\leq  \delta^{-2s_n}
\EE \left( \mathrm{Tr} \hspace{1pt} \left[\left(\hA^{<\vep_n }\right)^{2s_n}   \right]\right)\,,
\end{align}
where, for later purposes (c.f. Proposition \ref{lem:aboundofallevenpaths}), we assume that  the integer sequence $s_n>C\log n$ for some appropriate $C>0$, and  that $s_n$ is slowly varying\footnote{It is possible to allow $s_n$ to be some small power of $n$ (modulo the integer restriction), however, for convenience we assume it is slowly varying.} in $n$.
	Most of the rest of this section is devoted to proving Proposition \ref{lem:aboundofallevenpaths} from which we deduce that the right-hand side of \eqref{eq:frobenius2} goes to 0 as $n\to\ff$  for any $\delta>0$. Before we prove this result, we will prove some preliminary lemmas.

We write $\h X^{<}_{ij} := \h X_{ij}  \1{|X_{ij}| < n^{\beta}  }, 
\h X^{>}_{ij}:=\h X_{ij}-\h X^{<}_{ij}$,
so that in particular 
	\begin{align}\label{A trunc def}
	\h A^{<\vep_n }_{ij}=b_n^{-1}\, \h X^{<}_{ij}, \ \h A^{>\vep_n }_{ij}=b_n^{-1}\, \h X^{>}_{ij}\,,
\end{align}
where we recall $\vep_n=\frac{n^{\beta}}{b_n}$. 

By Lemma 13 in \cite{auffinger:arous:peche:2009}, we have for $n$ large enough that $|\E[\h A^{<\vep_n }_{ij}]|\le b_n^{-1} \widetilde{L}(n^{\beta}) n^{\beta (1-\alpha)}$, where $\widetilde{L}$ is the slowly varying function in $\eqref{regvar}$. Since $\|\E[ \hA^{<\vep_n }]\| =n |\E[\h A^{<\vep_n }_{ij}]| \to 0$ by \eqref{eq:betabounds} and the Potter bounds for slowly varying functions, we may assume without loss of generality that $\E[\h A^{<\vep_n }_{ij}]=-\E[\h A^{>\vep_n }_{ij}]=0$; compare also with equation (37) in \cite{auffinger:arous:peche:2009}.

We will show that the spectral norm of $\hA^{<\vep_n }=\frac{\h \bsX^{<}}{b_n}$ is bounded in probability by using \eqref{eq:frobenius2} and estimating $\EE ( \mathrm{Tr} \hspace{1pt} ( \h \bsX^{<})^{2s_n} )$ using a modification of the moment method described in \cite{auffinger:arous:peche:2009}. 

We begin with a standard calculation for the trace of some even power of a matrix. 
Let ${\mathfrak P}{={\mathfrak P}^{(n)}}$ be the set of ordered $(2s_n+1)$-tuples $(i_0, i_1,\dots, i_{2s_n-1},i_0)$ with the same first and last coordinates and $1\le i_j\le n$. Thus, ${\mathfrak P}$  can be viewed as the set of closed paths 
$${\cal P}=[i_0 \rightarrow i_1 \rightarrow i_2 \rightarrow \dots i_{2s_n-1} \rightarrow i_0]$$ 
 of length $2s_n+1$, in the set $\{1,2,3, \dots , n\}$.
To simplify notation, we define
$$\h \bsX^{<}({\mathcal P}):=\h X^{<}_{i_0i_1}\h X^{<}_{i_1i_2}\h X^{<}_{i_2i_3} \cdots  \h X^{<}_{i_{2s_n-2}i_{2s_n-1}}\h X^{<}_{i_{2s_n-1}i_0},$$
 so that       
\begin{align}\label{eq:traceandclosedpath0}
\EE \left( \mathrm{Tr} \hspace{1pt} \left( \h \bsX^{<}\right)^{2s_n}   \right)
&=\sum_{\mathfrak P} \EE\left( \h \bsX^{<}({\mathcal P})\right) .
\end{align}

\subsubsection{Spectral norm of $\hA^{<\vep_{n}}$: a review of the i.i.d. case}

It will be useful to first review the argument in \cite{auffinger:arous:peche:2009}, for the case where the matrix entries are i.i.d. in the upper triangle, and $2\le \alpha<4$. They first consider the contribution to \eqref{eq:traceandclosedpath0} of all closed {\it even} paths ({\it even} means that every ``edge'' $(i_k,i_{k+1})$ in the path appears an even number of times).

Using terminology which goes back to \cite{soshnikov1999universality}, for each even path 
$${\cal P}=[i_0 \rightarrow i_1 \rightarrow i_2 \rightarrow \dots i_{2s_n-1} \rightarrow i_0],$$
an instant ${t}\in \{1,2, \dots 2s_n-1\}$,  and corresponding vertex $i_t$, is said to be {\it marked} if the nonoriented edge $\{i_{{t}-1} , i_{t}\}$ where ${t}\in \{1,2, \dots 2s_n-1, 2s_n=0\}$ occurs an odd number of times up to (and including) instant ${t}$, otherwise it is said to be {\it unmarked}. Ignoring $t=0$ (which is assumed to be an unmarked instant), 
 it follows that the number of marked instants equals the number of unmarked instants. 
 For a given path $\cal P$, denote by $\cN_k$, the subset of $\{1,\dots ,n\}$ occurring $k$ times as a marked vertex where $0 \leq k \leq s_n$ and set $n_k:=|\cN_k|$. Any vertex in $\cN_k$ is said to have $k$ {\it self-intersections}. We say that $(n_0,n_1, \dots ,n_{s_n})$ is the {\it type} of path ${\cal P}$. From the definition of $n_k$, we get
\begin{equation}\label{nk sum}
\sum_{k=0}^{s_n} n_k=n \hspace{0.5cm} \text{and} \hspace{0.5cm} \sum_{k=0}^{s_n} kn_k=s_n.
\end{equation}
For given ${\cal P}$ define $\ell(ij)$ as the number of times the nonoriented edge $\{i,j\}$ appears in the path.
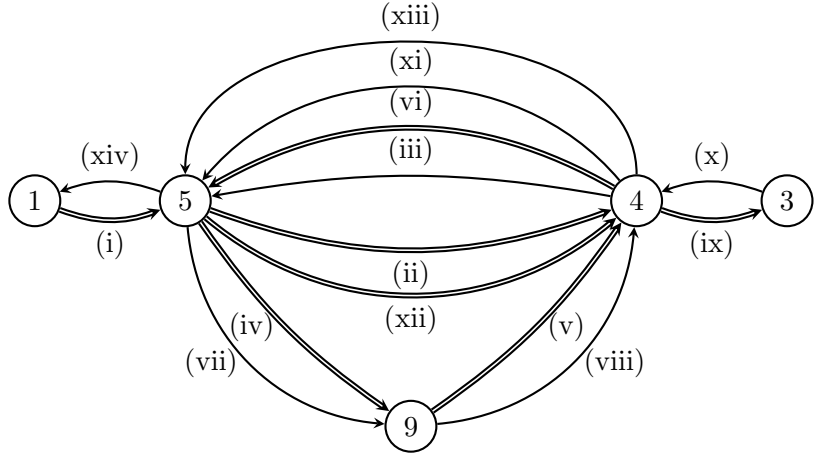
\begin{figure}[ht!]
	\centering
	\begin{tikzpicture}[>={stealth[black]}]
	\node[circle,thick,draw] (1) at (0,0) {1};
	\node[circle,thick,draw] (5) at (2,0) {5};
	\node[circle,thick,draw] (4) at (8,0) {4};
	\node[circle,thick,draw] (9) at (5,-3) {9};
	\node[circle,thick,draw] (3) at (10,0) {3};
	\path [->] (1) edge [bend right=20,thick,double] node[below] {(i)} (5);
	\path [->] (5) edge [bend right=20,thick,double] node[below] {(ii)} (4);
	\path [->] (4) edge [bend right=10,thick] node[above] {(iii)} (5);
	\path [->] (5) edge [thick,bend right=10,double] node[left] {(iv)} (9); 
	\path [->] (9) edge [bend right=10,thick,double] node[right] {(v)} (4);
	\path [->] (4) edge [bend right=30,thick,double] node[above] {(vi)} (5);
	\path [->] (5) edge [bend right=40,thick] node[left] {(vii)} (9); 
	\path [->] (9) edge [bend right=40,thick] node[right] {(viii)} (4);
	\path [->] (4) edge [bend right=20,thick,double] node[below] {(ix)} (3);
	\path [->] (3) edge [bend right=20,thick] node[above] {(x)} (4);
	\path [->] (4) edge [bend right=50,thick] node[above] {(xi)} (5);
	\path [->] (5) edge [bend right=40,thick,double] node[below] {(xii)} (4);
	\path [->] (4) edge [bend right=90,thick] node[above] {(xiii)} (5);
	\path [->] (5) edge [bend right=20,thick] node[above] {(xiv)} (1);	
	\end{tikzpicture}
	\caption{the target of $\Rightarrow$ becomes a marked vertex}
	\label{figure}
\end{figure}
\begin{example} \label{exampleofapath}
	Consider the following path ${\cal P}$ with $n=30$ and $2s_n=14$, and with $u,m$ denoting unmarked/marked instances (see Figure \ref{figure}):
	$$ {\cal P}=[1_u \rightarrow 5_m \rightarrow 4_m \rightarrow 5_u \rightarrow 9_m \rightarrow 4_m \rightarrow 5_m \rightarrow 9_u \rightarrow 4_u \rightarrow 3_m \rightarrow 4_u \rightarrow 5_u \rightarrow 4_m \rightarrow 5_u \rightarrow 1_u].$$
	Then,  $\cN_0=\{1,2,6,7,8,10, \dots ,30 \}, \cN_1=\{3,9\}, \cN_2=\{5\}, \cN_3=\{4\}$ and all other $\cN_k$'s are empty. Thus, the type of this path is  $(n_0,n_1, \dots ,n_{s_n})=(26,2,1,1,0, \dots ,0).$ Also, $\ell(4\  5)=6$. 
\end{example}

\noindent {\bf Notation:} {\it  For the rest of this proof, $L(n)$ denotes a generic slowly varying function of $n$, which may change from line to line.} 

In \cite[Lemma 15]{auffinger:arous:peche:2009}, the following moment bound for even paths was proved.

\begin{lemma}\label{lem:singletypepathbound}
Let $2\le\alpha<4$.	If $\{\h X_{ij}, i\le j\}$ are i.i.d. then
	 for an even path ${\cal P}$ of type $(n_0,n_1,\dots,n_{s_n})$, we have
	\begin{align}
	\EE\(\h \bsX^{<}({\cal P})\) \leq L^{n_1}(n^\beta) \prod_{k=2}^{s_n}  \big( L^k(n^\beta) n^{\beta (2k-(\alpha/2-1))}\big)^{n_k}.
	\end{align}
\end{lemma}

\begin{remark} \label{rem:1}
In \cite{auffinger:arous:peche:2009}, they assume $\alpha>2$ and their $\sigma^{2s_n}$ factor corresponds to the $k=1$ or $n_1$ component of the bound. More precisely,  $\sigma^{2s_n}$ is the truncated second-moment portion of the entries. In our setting we also include the case $\alpha=2$, where the truncated second moment is slowly varying but might tend to infinity. Thus the $\sigma^{2s_n}$ factor is replaced with a slowly varying function to the power $n_1$.
\end{remark}

In order to use the above lemma, we normalize the matrix in \eqref{eq:traceandclosedpath0} by $\frac{1}{n^{2/\alpha-\vep}}$ to get
\begin{align}\label{eq:traceandclosedpath}
\EE \left( \mathrm{Tr} \hspace{1pt} \left(\frac{1}{n^{2/\alpha-\vep}} \h \bsX^{<}\right)^{2s_n}   \right)
&=\sum_{\mathfrak P} \EE\left(\frac{1}{n^{2s_n(2/\alpha-\vep)}} \h \bsX^{<}({\mathcal P})\right) ,
\end{align} 
where
\begin{align}\label{eq: eps choice}
\vep<\min\left\{\frac{2}{\alpha}-\frac1 2,\, \frac{2}{\alpha}-\beta, \, \frac{1}{4}\left(\frac{8}{\alpha}-1-\beta(5-\frac{\alpha}{2})\right)\right\}.
\end{align}
We need $\vep< \frac{2}{\alpha}-\frac12$ to show that \eqref{contributionofallsimplepath}  goes to 0, while the other two bounds on $\vep$ are required for  \eqref{eq:2^2sn} which in turn is essential in showing that that $\| \hA^{<\vep_n } \|\cip 0$.

In the case where $\{\h X_{ij}, i\le j\}$ are i.i.d., \cite{auffinger:arous:peche:2009} goes on to show  that when $Z_e=Z_e(\vec{n})$ is the contribution of all even paths of type $\vec{n}=(n_0,n_1,\dots,n_{s_n})$  to \eqref{eq:traceandclosedpath},
\begin{align}\label{eq:ze bound} 
Z_e(\vec{n}) \leq \frac{(2s_n)!}{s_n! (s_n+1)!}n (2e)^{s_n} \prod_{k=2}^{s_n} \frac{1}{n_k!}\left[\frac{L^k(n^\beta) s_n^k n^{\beta(2k-(\alpha/2-1))}}{n^{4k/\alpha - 2k\vep -1}}\right]^{n_k} \frac{L^{n_1}(n^\beta)}{n^{n_1(4/\alpha-2\vep -1)}}.
\end{align}
\begin{remark} \label{rem:2} 
	In addition to the modification discussed in \Cref{rem:1}, the bound in \eqref{eq:ze bound} slightly differs from the analogous bound in \cite{auffinger:arous:peche:2009} by a factor of $(2e)^{s_n}$ which (due to its insignificance) seems to have been dropped from the line above (42) to (42) in \cite{auffinger:arous:peche:2009}.
\end{remark}

A path is {\it simple} if it has a type of form $(n_0,n_1,0, \dots , 0)=(n-s_n,s_n,0, \dots , 0)$ and is {\it intersecting} if $\sum\limits_{k=2}^{s_n} n_k>0$. In \cite{auffinger:arous:peche:2009}, the notation  $$Z_{e,s}:=Z_e(n-s_n,s_n,0, \dots , 0)$$ is used for the contribution to \eqref{eq:traceandclosedpath} of all simple even paths. For this, \cite{auffinger:arous:peche:2009} obtains the bound
\begin{equation}\label{contributionofallsimplepath}
Z_{e,s} \leq \frac{(2s_n)!}{s_n! (s_n+1)!}n \frac{L^{s_n}(n^\beta)}{n^{{s_n}(4/\alpha-2\vep -1)}}
\end{equation}
and argues that this goes to 0 as $n\to\infty$ if {$\vep<\frac2{\alpha}-\frac12$.}
For the contribution of all intersecting even paths,  denoted by $$Z_{e,i}:=\sum_{\vec{n}:\sum\limits_{k\ge 2} n_k>0} Z_e(\vec{n})\, ,$$ one obtains
\begin{equation}\label{eq:auffingercontributionofallnonsimple}
Z_{e,i} \leq o(1) \frac{(2s_n)!}{s_n! (s_n+1)!}n (2\sqrt{2e})^{2s_n} .
\end{equation} 
\begin{remark}
	In addition to the factor $(\sqrt{2e})^{2s_n}$ discussed in \Cref{rem:2},
	we have included an extra $2^{2s_n}$ in the bound of $Z_{e,i}$ when compared to \cite{auffinger:arous:peche:2009}. The explanation for this additional extra factor is given below \eqref{boundofinsideofprod}.
\end{remark}

\subsubsection{Spectral norm of $\hA^{<\vep_{n}}$: the $m$-dependent case}
To extend the above technique to the {\it m}-dependent case, we will map each path ${\cal P}$ to a new path ${\psi(\cal P)}$ which we now describe.
For each path 
$${\cal P}=[i_0 \rightarrow i_1 \rightarrow i_2 \rightarrow \dots i_{2s_n-1} \rightarrow i_0],$$ 
we form { equivalence classes} (depending on the path) of vertices in the path as follows. First break the path into ordered pairs
$\{(i_0,i_1), (i_1,i_2),\ldots, (i_{2s_n-1},i_{0})\}$ and extend this set of ordered pairs to a set $\bar {\cal P}$ which includes all reflections, i.e., if $(x,y)$ is in this set then we put 
\begin{align}\label{def: bar P}
(x,y)\in \bar {\cal P}\quad\text{and}\quad (y,x)\in\bar{\cal P}.
\end{align}
 For $(x,y),(u,v)\in \bar {\cal P}$, if $|(x,y),(u,v)|:=\max(|x-u|,|y-v|)\le m$, then we put $x$ and $u$ into the same class and $y$ and $v$ into the same class.  After extending this by transitivity, we get an equivalence relation between vertices. 
 Consequently, if $(x,u),(y,v),(z,w)$ are elements of $\bar {\cal P}$ with $x,y,z$ all being in the same class, then $x,y,z$ are not necessarily within pairwise distance $m$ of each other. Indeed, it may be the case that $|x-y|\le m$ and $|y-z|\le m$ but that $|x-z|>m$. However, by transitivity, if $x$ and $z$ are in the same class, then there is some chain of vertices $y_1,\ldots,y_k$ in that same class such that 
	\begin{align}\label{chain}
	|x-y_1|\le m,\ |y_1-y_2|\le m,\ \ldots,\ |y_k-z|\le m. 
	\end{align}

Now, let all the vertices in a given class be represented by the minimum vertex of that class. Denote by $i_{k}^{\prime}$ the new value of vertex $i_{k}$ according to this procedure. We also denote the class of vertices containing $i_k^{\prime}$ by $G_{i_k^{\prime}}=G_{i_k^{\prime}}({\cal P})$. 
Let $\psi$ map a path ${\cal P}$ to the new path of the form
 $${\psi(\cal P)}=\mathcal{P}'=[i_{0}^{\prime} \rightarrow i_{1}^{\prime} \rightarrow i_{2}^{\prime} \rightarrow \dots i_{2s_n-1}^{\prime} \rightarrow i_{0}^{\prime}].$$ 
We define $\cN_k'$ and $n_k'$ for the path $\psi(\cal{P})$ analogously to the quantities $\cN_k$ and $n_k$ for the path $\cal{P}$.
Furthermore, we write $\ell^{\prime}(i^{\prime}j^{\prime})$ for the number of appearances of the nonoriented edge $\{i^{\prime},j^{\prime}\}$ in $ \psi({\cal P})$.

By construction, the vertices in $ \psi({\cal P})$ are separated by at least $m+1$. Recalling that $\bsX^{\infty}$ is $m$-dependent, we see that
$$\EE[\h \bsX^{<}({\mathcal \psi({\cal P})})] = \prod_{\{i^{\prime},j^{\prime}\}: \{i',j'\}\in \psi({\cal P})} \EE  \left[  (\h X^{<}_{i^{\prime}j^{\prime}})^{\ell^{\prime}(i^{\prime}j^{\prime})} \right]\,.$$

\begin{example} \label{exampleofprimepath}
	We recall from Example \ref{exampleofapath} 	
		$$ {\cal P}=[1_u \rightarrow 5_m \rightarrow 4_m \rightarrow 5_u \rightarrow 9_m \rightarrow 4_m \rightarrow 5_m \rightarrow 9_u \rightarrow 4_u \rightarrow 3_m \rightarrow 4_u \rightarrow 5_u \rightarrow 4_m \rightarrow 5_u \rightarrow 1_u].$$
	Now let us assume that $m=1$, i.e., that there is $1$-dependence. Then $3,4,5$ all belong to the same class since, for instance, $|(3,4),(4,5)|\le 1$.
	We get the relations $$1'=1, \ 3'=3, \ 4'=3, \ 5'=3, \ 9'=9,$$ and the classes $G_{1}=\{1\}, G_{3}=\{3,4,5\}, $ and $ G_{9}=\{9\}$. From these new vertices under the map $\psi$, we have 
	\begin{align*}
	&{\psi(\cal P)}= [ 1_u \rightarrow 3_m \rightarrow 3_m \rightarrow 3_u \rightarrow 9_m \rightarrow 3_u \rightarrow 3_u \rightarrow 9_m \\
	&\rightarrow 3_u \rightarrow 3_u \rightarrow 3_m \rightarrow 3_u \rightarrow 3_m \rightarrow 3_u\rightarrow 1_u].
	\end{align*}
Then have {$\cN_0'=\{1, 2, 4 \dots, 8, 10, \dots, 30 \}, \cN_2'=\{ 9 \}, \cN_3
	'=\{3\} $} with all other $\cN_k'$'s empty.  We also see that, for instance, $\ell'(3\,9)=4$.
\end{example}


\begin{lemma}\label{lem:newpathinq}
	Assume $\{\h X^{<}_{ij}, i\le j\}$ are {\it m}-dependent. For all ${\cal P}$ we have 
\begin{equation*}
 \EE| \h \bsX^{<}({\mathcal P})|  \leq \EE|\h \bsX^{<}({\mathcal \psi({\cal P})})|\,.
\end{equation*}

	\begin{proof}	
		 Let  $\ell(ij)$ be the number of appearances of the nonoriented edge $\{i,j\}$ in the path ${\cal P}$ and $\ell^{\prime}(i^{\prime}j^{\prime})$ the number of appearances of the nonoriented edge $\{i^{\prime},j^{\prime}\}$ in $ \psi({\cal P})$. From the definition, we have $\ell^{\prime}(i^{\prime}j^{\prime})=\sum_{ i\in G_{i^{\prime}} \text{ and } j\in G_{j^{\prime}} } \ell(ij)$. 
		Then, for each fixed $\{i',j'\}$ in $\psi({\cal P})$, the general H\"older inequality along with stationarity gives	
		 \begin{align*}
		 \EE\left(\prod_{{\{i,j\}:i\in G_{i^{\prime}}} \text{ and } j\in G_{j^{\prime}} } \left|\left(\h X_{ij}^{<}\right)^{\ell(ij)}\right|\right) 
		 &\leq \prod_{{\{i,j\}:i\in G_{i^{\prime}} \text{ and } j\in G_{j^{\prime}} }} \left( \EE \left|\left(\h X_{ij}^{<}\right)^{\ell(ij)}\right|^{\frac{\ell'(i',j')}{\ell(i,j)}} \right)^{\frac{\ell(i,j)}{\ell'(i',j')}} \\ \\
		 &= \EE \left|\h X_{i'j'}^{<}\right|^{\ell'(i'j')}.
		 \end{align*}
		 Taking the product over all $\{i',j'\}\in \psi({\cal P})$, we get
		 	\begin{equation*}
		  \EE \left(\left|\prod_{\{i,j\}:\{i,j\}\in {\cal P}} \left(  \h X^{<}_{ij}\right)^{\ell(i j)}\right|\right) \leq   
		 	   \prod_{\{i^{\prime},j^{\prime}\}: \{i',j'\}\in \psi({\cal P})} \EE  \left(\left|  \h X^{<}_{i^{\prime}j^{\prime}}\right|\right)^{\ell^{\prime}(i^{\prime}j^{\prime})} 
		 	\end{equation*}
		 	which is precisely what we wanted to show.
 	\end{proof}   	
\end{lemma}


Lemma \ref{lem:newpathinq} implies that the contribution of each path can be bounded by the contribution of its induced path $\mathcal{P}'$ under the map $\psi$. Moreover, for each fixed path $\mathcal{P}'$,  if we count the number of paths ${\cal P}$ satisfying $\psi(\mathcal{P})=\mathcal{P}'$ and denote this number by $\left|\psi^{-1}(\mathcal{P}')\right|$, Lemma \ref{lem:newpathinq} gives us that   
\begin{equation}\label{eq:dsaaa}
\sum_{{\mathcal P\in \mathfrak{ P} }} \EE\left(  \left| \h \bsX^{<}({\mathcal P})\right|\right) \leq  \sum_{{\mathcal P'\in \mathfrak{ P}_m }} \left|\psi^{-1}(\mathcal{P}')\right| \EE\left(  \left|\h \bsX^{<}({\mathcal P'})\right|\right)
\end{equation}
where $\mathfrak{ P}_m$ is the collection of paths ${\cal P'}$ such that any two ordered pairs  $(x,y),(u,v)\in \bar{\cal P'}$ (see \eqref{def: bar P}) have distance $|(x,y),(u,v)|>m$. Note that, under the map $\psi$, every path $\cal P\in\mathfrak{P}$ maps to some path in ${\cal P'} \in\mathfrak{ P}_m$. Moreover, if ${\cal P'}\in\mathfrak{ P}_m$, then any two distinct random variables in the product $\h \bsX^{<}({\mathcal P'})$ are independent. 

In order to get a bound on the right-hand side of \eqref{eq:dsaaa}, we consider random variables $\{Y_{ij}\}$ which satisfy 
\begin{align}\label{eq:Ydef}
	Y_{ij} \eid  X_{ij} \quad \text{ and } \quad \{Y_{ij}\}_{1\leq i,j\leq n} \text{ are independent.} 
\end{align}
Since $|(i_{k},i_{k+1})-(i_{l},i_{l+1})|> m$ for distinct $(i_{k},i_{k+1}), (i_{l},i_{l+1}) \in \mathcal{P}' $ with ${\cal P'} \in\mathfrak{ P}_m$, we have
\begin{align*}
\sum_{{\cal P'} \in\mathfrak{ P}_m} \EE\left(  \left|\h \bsX^{<}({\mathcal P'})\right|\right) 
&= \sum_{{\cal P} \in\mathfrak{ P}_m} \EE\left(  \left| \h \bsY^{<}({\mathcal P})\right|\right) 
\leq  \sum_{{\cal P} \in\mathfrak{ P}}\EE\left(  \left| \h \bsY^{<}({\mathcal P})\right|\right).
\end{align*}
The argument of \cite{auffinger:arous:peche:2009} in Eqs. (46)-(50) (which relies on the work of \cite{peche2007wigner}) implies that, when the matrix elements have mean zero,
\begin{align*}
\sum_{\mathfrak{P}}\EE\left(  \left| \h \bsY^{<}({\mathcal P})\right|\right)
&=\sum_{\mathfrak{P}_{\text{even}}} \EE\left(  \left| \h \bsY^{<}({\mathcal P})\right|\right) +\sum_{\mathfrak{P}_{\text{odd}}} \EE\left(  \left| \h \bsY^{<}({\mathcal P})\right|\right)\\
&=(1+o(1)) \,\sum_{\mathfrak{P}_{\text{even}}} \EE\(  \h \bsY^{<}({\mathcal P})\)\,,
\end{align*}
where $\mathfrak{ P_{\text{even}}}$ is the set of all even paths and $\mathfrak{ P_{\text{odd}}}$ is the set of all odd paths. For large enough $n$, we get the following bound
\begin{equation} \label{eq:m-dependentpathsum}
\sum_{\substack{\mathfrak{ P} }} \EE\left(  \h \bsX^{<}({\mathcal P})\right) \leq  2 \max_{\mathcal{P}' \in\mathfrak{ P}_m} \left|\psi^{-1}(\mathcal{P}')\right|     \sum_{\mathfrak{P}_{\text{even}}} \EE\left(  \h \bsY^{<}({\mathcal P})\right).
\end{equation}

We are left to estimate $\left|\psi^{-1}(\mathcal{P}')\right|$. Before proving an upper bound, we consider a simple example. 

\begin{example}
Recall Example \ref{exampleofprimepath} with $m=1$.  We have a path of length $14$ such that 
$$ {\cal P}'=[1\rightarrow 3\rightarrow 3\rightarrow 3\rightarrow 9\rightarrow 3\rightarrow 3 \rightarrow 9\rightarrow 3\rightarrow 3 \rightarrow 3 \rightarrow 3\rightarrow 3\rightarrow 3\rightarrow 1]. $$
Since the original path $\cal P$ is closed,  $G_1$ must be $\{1\}$. Next, consider the vertex $3$ in the path $\mathcal{P}'$. Since it appears $11$ times, $G_3$ has at most $11$ distinct elements which we may order from smallest to largest. Thus, the possible number of different sets $G_3$ associated with different ${\cal P}$'s such that $\psi(\cal P)=\cal P'$ is bounded by $(m+1)^{11-1}=2^{10}$. For example, it could be $G_3=\{3\}$ or all 11 elements might be distinct as in $G_3=\{3,4,5,6,7,8,9,10,11,12,13,14\}$ (ignoring for the time being that in fact $9\notin G_3$ since 9 appears in $\cal P'$). Finally, for $G_9$, we obtain at most $2$ different sets since 9 only appears twice; either $G_9=\{9\}$ or $G_9=\{9,10\}$.

For the purposes of a crude upper bound let us continue to ignore the fact that we must have $9\notin G_3$ and suppose that
\begin{align*}
G_1&=\{1\}\,,G_3=\{3,4,5,6,7,8,9,10,11,12,13,14\}\,,G_9=\{9,10\}\,.
\end{align*}
Then the possible number of $\mathcal{P}$'s associated to this choice of $G_1,G_3,G_9$ is bounded by $11!2!1!\le 14!$. Therefore, the number of different paths $\cal P$ which map to $\cal P'$ above 
 is at most $2^{10}\cdot 2 \cdot14!$. 
\end{example}

\begin{lemma}\label{lem:count}
		For any closed even path $\mathcal{P}' \in\mathfrak{ P}_m$ of length $2s_n$, there exist at most $(m+1)^{2s_n}(2s_n)!$ different paths $\mathcal{P}$ which map to $\mathcal{P}'$ under $\psi$, i.e., we have the bound
		$$\left|\psi^{-1}(\mathcal{P}')\right|\le (m+1)^{2s_n}(2s_n)!\,\,.$$
	\begin{proof}
		Let $\mathcal{P}'=[i_0\rightarrow i_1 \rightarrow \cdots \rightarrow i_{2s_n-1}\rightarrow i_0]$ be a closed path of length $2s_n$,
			and denote by $\#(i_k)$ the number of times that $i_k$ appears in $\mathcal{P}'$.
		
For a given $i_k$, the number of possible $G_{i_k}$ is bounded by \mbox{$(m+1)^{\#(i_k)}$.}
Let $\{j_1,\ldots,j_k\}$ denote the set of distinct vertices in $\cal P'$. 
Given a fixed instance of the sets $G_{j_1},\ldots,G_{j_k}$, the number of ways we may assign vertices in these sets to the various positions $1,\ldots,{2s_n-1}$ in $\mathcal{P}'$ is at most $\prod_{l=1}^k \#(j_l)!\,$. 
Thus we obtain the bound
$$\prod_{l=1}^k(m+1)^{\#(j_l)}\#(j_l)!\le (m+1)^{2s_n}(2s_n)!\,.$$
	\end{proof}
\end{lemma}

{ \begin{proposition}\label{lem:aboundofallevenpaths}
	Assume $\{\h X_{ij}, i\le j\}$ are {\it m}-dependent. Let $C>0$ be arbitrary and assume that the slowly varying sequence $(s_n)$ satisfies $s_n>C\log n$.
Then there exist positive constants $\gamma(C), \eta(C)$ such that, for any $\gamma>\gamma(C)$ and sufficiently large $n$, it holds
	$$\frac{\EE \left( \mathrm{Tr} \hspace{1pt} (\frac{1}{n^{2/\alpha-\vep}} \h \bsX^{<})^{2s_n}   \right)}{ ( (m+1) (8+\gamma)\sqrt{2e}   )^{2s_n}  } \leq \exp(-\eta(C) s_n)\,.$$
\end{proposition}	

Before presenting the proof of Proposition~\ref{lem:aboundofallevenpaths}, we show how it is used to derive $\|\hA^{<\vep_{n}}\| \cip 0$, as $n\to \infty$, in the case $2\leq\alpha<4$. Recall that $\vep$ satisfies \eqref{eq: eps choice}. For $C>0$ let the slowly varying sequence $(s_n)$ satisfy $s_n>C\log n$ and assume that $\gamma$ is sufficiently large. By \eqref{eq:frobenius2}, an application of Proposition~\ref{lem:aboundofallevenpaths} yields, for any $\delta>0$, that
\begin{align*}
\PP\left( \big\| \hA^{<\vep_n } \big\| > \delta\right) 
&\leq  \Big( \underbrace{\frac{ n^{2/\alpha-\vep}(m+1) (8+\gamma)\sqrt{2e}}
{b_n \delta}}_{\to 0} \Big)^{2s_n}
\frac{\EE \left( \mathrm{Tr} \hspace{1pt} (\frac{1}{n^{2/\alpha-\vep}} \h \bsX^{<})^{2s_n}   \right)}{ ( (m+1) (8+\gamma)\sqrt{2e}   )^{2s_n}  } \to 0\,,
\end{align*}
as $\nto$, which establishes $\|\hA^{<\vep_{n}}\| \cip 0$.

\begin{proof}[Proof of Proposition~\ref{lem:aboundofallevenpaths}]
	Let $\bsY=(Y_{ij}){\in \R^{n\times n}}$ be as in \eqref{eq:Ydef} and $C>0$. {From \eqref{eq:m-dependentpathsum}, and Lemma \ref{lem:count},} we obtain
	\begin{equation} \label{eq:contributionm-dependentpath}
	\frac{\EE \left( \mathrm{Tr} \hspace{1pt} (\frac{1}{n^{2/\alpha-\vep}} \h \bsX^{<})^{2s_n}   \right)}{ ( (m+1) (8+\gamma)\sqrt{2e}   )^{2s_n}  } 
	\leq \frac{2 (2s_n)!}{((8+\gamma)\sqrt{2e}   )^{2s_n}  } \sum_{\mathfrak{P}_{\text{even}}} \EE\left(\frac{1}{n^{2s_n(2/\alpha-\vep)}} \h \bsY^{<}({\mathcal P})\right).
	\end{equation}
	Moreover, since $\bsY^{<}({\mathcal P})$ is a product of independent random variables, we may apply the inequality \eqref{eq:ze bound} which we recall for the reader's convenience
	\begin{align*}
	Z_e(\vec{n}) \leq \frac{(2s_n)!}{s_n! (s_n+1)!}n (2e)^{s_n} \frac{L^{n_1}(n^\beta)}{n^{n_1(4/\alpha-2\vep -1)}}\prod_{k=2}^{s_n} \frac{1}{n_k!}\left[\frac{L^k(n^\beta) s_n^k n^{\beta(2k-(\alpha/2-1))}}{n^{4k/\alpha - 2k\vep -1}}\right]^{n_k} ,
	\end{align*}
where $Z_e(\vec{n})$ denotes the contribution of all even paths of type $\vec{n}=(n_0,n_1,\dots,n_{s_n})$ to the sum
$$\sum_{\mathfrak{P}_{\text{even}}} \EE\left(\frac{1}{n^{2s_n(2/\alpha-\vep)}} \h \bsY^{<}({\mathcal P})\right).$$
Writing $L$ instead of $L(n^{\beta})$, we deduce that
\begin{align}\label{eq:deduced}
&\frac{\EE \left( \mathrm{Tr} \hspace{1pt} (\frac{1}{n^{2/\alpha-\vep}} \h \bsX^{<})^{2s_n}   \right)}{ ( (m+1) (8+\gamma)\sqrt{2e}   )^{2s_n}  } 
\leq \frac{2 (2s_n)!}{((8+\gamma)\sqrt{2e}   )^{2s_n}  } \sum_{\vec{n}}	Z_e(\vec{n}) \nn \\
& \le \frac{2 (2s_n)!}{((8+\gamma)\sqrt{2e}   )^{2s_n}  } \frac{(2s_n)!\,n (2e)^{s_n}}{s_n! (s_n+1)!}\sum_{\vec{n}}
 \frac{L^{n_1}}{n^{n_1(4/\alpha-2\vep -1)}} \frac{1}{\varphi(\vec{n})!}\prod_{k=2}^{s_n}\left[\frac{L^k s_n^k n^{\beta(2k-(\alpha/2-1))}}{n^{4k/\alpha - 2k\vep -1}}\right]^{n_k}, 
\end{align}
where $\varphi(\vec{n}):= \min_{i\ge 2} n_i$.
	We have 
	\begin{align}
	(2s_n)! \prod_{k=2}^{s_n} \left[\frac{L^k s_n^k n^{\beta(2k-(\alpha/2-1))}}{n^{4k/\alpha - 2k\vep -1}}\right]^{n_k}  
	&<  2^{2s_n} \prod_{k=2}^{s_n} \left[\frac{L^k s_n^{2k} s_n^k n^{\beta(2k-(\alpha/2-1))}}{n^{4k/\alpha - 2k\vep -1}}\right]^{n_k}, \label{bound of ze}
    \end{align}
    where we used that $(2s_n)! < (2s_n)^{2s_n} = 2^{2s_n}\cdot s_n^{2n_1}\cdot \prod_{k=2}^{s_n} \(s_n^{2k}\)^{n_k}$ by \eqref{nk sum}.
    
    Recall that $s_n$ is slowly varying.
    Note that since $\vep<2/\alpha-\beta$, by \eqref{eq: eps choice}, the largest factor in closed parentheses in the product on the right-hand side in \eqref{bound of ze} occurs when $k=2$. Thus, we obtain
    \begin{align}\label{boundofinsideofprod}
    \prod_{k=2}^{s_n} &\left[\frac{L^k s_n^{3k} n^{\beta(2k-(\alpha/2-1))}}{n^{4k/\alpha - 2k\vep  -1}}\right]^{n_k} 
    \leq \prod_{k=2}^{s_n} \left[\frac{L^2 s_n^6 n^{\beta(4-(\alpha/2-1))}}{n^{8/\alpha - 4\vep  -1}}\right]^{n_k} \nn\\
    &= \left[\frac{L^2 s_n^6 n^{\beta(5-\alpha/2)}}{n^{8/\alpha - 4\vep  -1}}\right]^{n-n_1}
 \leq \left[\frac{L^2 s_n^6 n^{\beta(5-\alpha/2)}}{n^{8/\alpha - 4\vep  -1}}\right]^{\min_{i\geq 2} n_i} \,,
    \end{align}
    where in the last inequality we used, from \eqref{eq: eps choice},
     that $\vep<\frac{1}{4}\left(\frac{8}{\alpha}-1-\beta(5-\frac{\alpha}{2})\right)$ and thus all the factors tend to zero by the Potter bounds.
		
		Since $L n^{-(4/\alpha-2\vep -1)}\to 0$, we deduce from \eqref{eq:deduced}, \eqref{bound of ze} and \eqref{boundofinsideofprod} that {for sufficiently large $n$}
		\begin{align*}
\frac{\EE \left( \mathrm{Tr} \hspace{1pt} (\frac{1}{n^{2/\alpha-\vep}} \h \bsX^{<})^{2s_n}   \right)}{ ( (m+1) (8+\gamma)\sqrt{2e}   )^{2s_n}  } 
& \leq \frac{2\, n (2e)^{s_n}2^{2s_n}}{((8+\gamma)\sqrt{2e}   )^{2s_n}  } \frac{(2s_n)!}{s_n! (s_n+1)!}\sum_{\vec{n}}
  \frac{1}{\varphi(\vec{n})!}\left[\frac{L^2 s_n^6 n^{\beta(5-\alpha/2)}}{n^{8/\alpha - 4\vep  -1}}\right]^{\varphi(\vec{n})}.
\end{align*}
    
		The number\footnote{We note that this factor (say $2^{2s_n}$) does not appear in (45) of \cite{auffinger:arous:peche:2009}. However, their proof is unaffected by such a factor.} of intersecting types $\vec{n}$ which map to the same $\varphi(\vec{n})$-value is trivially bounded by
		the number of partitions
 of $n$ into $n=n_0+n_1+\dots +n_{s_n}$ which is in turn bounded by $2^{2s_n}$. In conjunction with $\tfrac{(2s_n)!}{s_n! (s_n+1)!}\sim \tfrac{4^{s_n}}{\sqrt{\pi} s_n^{3/2}}$ as $n \to \infty$, {we get for sufficiently large $n$} that
	\begin{align}\label{eq:2^2sn}
	\frac{\EE \left( \mathrm{Tr} \hspace{1pt} (\frac{1}{n^{2/\alpha-\vep}} \h \bsX^{<})^{2s_n}   \right)}{ ( (m+1) (8+\gamma)\sqrt{2e}   )^{2s_n}  } 
& \leq \frac{2\, n 2^{2s_n}}{(8+\gamma )^{2s_n}  } \frac{4^{s_n}}{\sqrt{\pi} s_n^{3/2}}2^{2s_n}\sum_{M=0}^{\infty}
  \frac{1}{M!}\left[\frac{L^2 s_n^6 n^{\beta(5-\alpha/2)}}{n^{8/\alpha - 4\vep  -1}}\right]^M \nn \\
	&= \left(\frac{8}{8+\gamma}\right)^{2s_n} \frac{2n}{\sqrt{\pi} s_n^{3/2}} \,(1+o(1))\,.
	\end{align}
{The sum over $M$ on the right-hand side is $1+o(1)$. The 1 comes from $M=0$, while the other terms are $o(1)$ which,  as already noted below \eqref{boundofinsideofprod}, is due to the Potter bounds.} Using $s_n>C\log n$ , we conclude for a suitable constant $\gamma(C)>0$ and any $\gamma>\gamma(C)$ that there exists an $\eta>0$ such that
$$  \left(\frac{8}{8+\gamma}\right)^{2s_n} n \le \exp(-\eta s_n).$$
This establishes the claim of the proposition.
\end{proof}}

 
\subsubsection{ Analysis of $\hA^{> \vep_{n}}$ }\label{sec:second trunc}

 In order to extract from $\hA^{> \vep_{n}}$ the blocks with the largest contribution towards the extreme eigenvalues, we require another truncation level defined as 
\begin{equation*}
\tvep_{n}:=b_{n}^{(\kappa-1)/2} \qquad \text{for }\ \kappa>\eta>\alpha\beta-1\,, \ \ \eta \in (1/3,1).
\end{equation*}
Now, since $\vep_n=n^\beta/b_n$, one can calculate that $\eta>\alpha\beta-1$ implies that $\tvep_n>\vep_n$. 
Next, we set the two auxiliary sequences
\begin{align}\label{eta exponent}
k_n= n^{\eta}\quad\text{and}\quad r_n= n^{1-\eta}
\end{align}
for our number of blocks in a row and block sizes, respectively,
and define
\begin{align}\label{eq:defofpbpm}
P_b &:=\PP(\|\hB_{11}\|_{\max}>\vep_{n})=\PP\(\max_{\substack{ i\in [1, r_n] \, \text{and} \\  j  \in[1, r_n]}} \left| \h X_{ij}\right|   > {n^{\beta}}\), \nn\\
P_m &:=\PP \(\max_{\substack{ i\in [1, r_n] \, \text{and} \\  j  \in[r_n -m, r_n+m]}} \left| \h A_{ij}\right|   > \vep_{n}\)=\PP \(\max_{\substack{ i\in [1, r_n] \, \text{and} \\  j  \in[r_n -m, r_n+m]}} \left| \h X_{ij}\right|   > {n^{\beta}}\).
\end{align}
The basic union bound for maxima tells us that for sufficiently large $n$ and for some slowly varying $L(n)$, 
\begin{align}\label{lem:pbpm}
	P_{b}\le \frac{L(n)r_{n}^{2}}{{n^2\vep_{n}^\alpha}}\quad\text{and}\quad 
	P_{m}\le  \frac{L(n)mr_{n}}{{n^2\vep_{n}^\alpha}}.
\end{align} 
{ Analogously,} we obtain the following bounds for $n$ sufficiently large,       
\begin{align}\label{lem:pbpm2}
	&\PP(\|\hB_{11}\|_{\max}>\tvep_{n})=\PP\left(\max_{\substack{ i\in [1, r_n] \, \text{and} \\  j  \in[1, r_n]}} \left| \h X_{ij}\right| >b_{n}\tvep_{n}\right)
	\leq  \frac{L(n)r_{n}^{2}}{n^2\tvep_{n}^{\alpha}}\,, \nn\\
	&\PP \(\max_{\substack{ i\in [1, r_n] \, \text{and} \\  j  \in[r_n -m, r_n+m]}} \left| \h A_{ij}\right|   > \tvep_{n}\)=\PP \(\max_{\substack{ i\in [1, r_n] \, \text{and} \\  j  \in[r_n -m, r_n+m]}} \left| \h X_{ij}\right|   > b_{n}\tvep_{n}\)
	\leq \frac{mL(n)r_{n}}{n^2\tvep_{n}^{\alpha}}.
\end{align}

\begin{lemma}\label{lem:2blocks} Let $\alpha\in [2,4)$.  If $\eta>2/3$, then $\lim_{\nto} \P(S_1^{\tvep_{n}})=1$, where $S_1^{\tvep_{n}}$ is the complement of the set
\begin{equation*}
\Big\{\exists 1\le i,j,k\le k_n  \, \, \text{ with } j\neq k \text{ such that }\,   \|\hB_{ij} \|_{\max}>\tvep_{n}, \,  \,\|\hB_{ik} \|_{\max}>\tvep_{n}   \Big\}\, .
\end{equation*}
We also have $\lim_{\nto} \P(S_2^{\tvep_{n}})=1$, where $S_2^{\tvep_{n}}$ is the complement of the set
\begin{equation*}
\Big\{\exists 1\le i\le k_n  \, \text{ such that }\,   \|\hB_{ii} \|_{\max}>\tvep_{n} \Big\}\, .
\end{equation*}
\end{lemma}
\begin{proof}
	Following the same argument in the proof of Lemma \ref{lem:En}, using \eqref{lem:pbpm2} instead of Proposition \ref{prop:block:conv}, we obtain 
	\begin{align*}
	&\PP\(\exists i,j,k\,  \text{ with } j+1<k \, \text{ such that }\,   \|\hB_{ij} \|_{\max}>\tvep_{n}, \,  \,\|\hB_{ik} \|_{\max}> \tvep_{n}  \)\\
	&\leq k_{n}^{3}\left[ \frac{L(n)r_{n}^{2}}{n^2\tvep_{n}^{\alpha}}\right]^{2} \leq  k_{n}^{3}O( L(n)n^{-4+4(1-\eta)-(2\kappa-2)}),
	\end{align*}
	and
	\begin{align*}
	&\PP\(\exists i,j,k \, \text{ with } j+1=k \, \text{ such that }\,   \|\hB_{ij} \|_{\max}>\tvep_{n}, \,  \,\|\hB_{ik} \|_{\max}>\tvep_{n}   \)\\
		&\le k_{n}^{2}\left(\left[ \frac{L(n)r_{n}^{2}}{n^2\tvep_{n}^{\alpha}}\right]^{2}+\frac{mL(n)r_{n}}{n^{2}\tvep_{n}^{\alpha}}\right) \le k_{n}^{2}\frac{mr_{n}}{n^{2}\tvep_{n}^{\alpha}}\left(L(n)+\frac{L(n)r_{n}^{3}}{n^2\tvep_{n}^{\alpha}}\right)\\
	&\leq  k_{n}^{2}O(L(n)n^{ -2+1-\eta-(\kappa-1)}).
	\end{align*}
	Since $\eta>2/3$ and $\kappa>\eta$, the right-hand sides go to zero as $n\rightarrow \infty$ in both inequalities. The proof of $\lim_{\nto} \P(S_2^{\tvep_{n}})=1$ is analogous.
\end{proof}

The above lemma implies that $\P(S^{\tvep_{n}})\to 1$, where $S^{\tvep_{n}}:=S_1^{\tvep_{n}} \cap S_2^{\tvep_{n}}$. The proof of Lemma \ref{lem:factor2} also shows the following result.

\begin{lemma}\label{lem:factor22}
Let $\eta>2/3$ and $j\le \text{rank}(\hA^{>\tvep_{n}})/2$. On the set $S^{\tvep_{n}}$, the $j$th largest and $j$th smallest eigenvalues of the matrix $\hA^{>\tvep_{n}}$ are given by $\lambda_j$ and $-\lambda_j$, respectively, where $\lambda_j$ is the $j$th largest value (counted with multiplicity) in the set
 $$\bigcup_{(k,l): \|\hB_{kl}\|_{\max}>\tvep_{n}, k<l} \{\sigma_1(\hB_{kl}^{>\tvep_{n}}), \ldots, \sigma_{r_n}(\hB_{kl}^{>\tvep_{n}})\}\,,$$
where $\sigma_i(\hB_{kl}^{>\tvep_{n}})$ denotes the $i$th largest singular value of $\hB_{kl}^{>\tvep_{n}}$.
\end{lemma}

Although we cannot directly prove an analogue of Lemma \ref{lem:2blocks} for the $\vep_{n}$-truncation, for $\eta>5/6$ we can show that the block matrix $\hA^{>\vep_{n}}$ has a bounded number of blocks whose max-norms lie in $[\vep_{n}, \tvep_{n}]$, for any given row. To this end we define the following events:
\begin{equation}\label{eq:ElFl}
E_{l}:=\{\|\hB_{1l}\|_{\max}>\vep_{n}\} \quad\text{and}\quad  
F_{l}:=\left\{\max_{(i,j)\in I_{l}} \left|\h A_{i j}\right|>\vep_{n}\right\} \text{ for }l=1,\cdots,k_{n},
\end{equation}
where 
\begin{equation*}
I_{l}:=\{(i,j):i\in [1,r_n],\,\vert j-lr_{n}\vert\leq m \} \quad\text{and}\quad  I_{0}\equiv I_{k_n}\equiv\emptyset.
\end{equation*}
We first bound the probability of having consecutive nonzero blocks in $\hA^{>\vep_{n}}$.
\begin{lemma}\label{lem:consecblock}
	We have the following bounds:
	\begin{align*}
	\PP(E_{l}\cap E_{l+1})				&\leq P_{m}+P_{b}^{2},&\quad \forall 1\leq l\leq k_{n}-1,\\
	\PP(E_{l}\cap E_{l+1}\cap E_{l+2})	&\leq P_{m}^{2}+2P_{b}P_{m}+P_{b}^{3},&\quad \forall 1\leq l\leq k_{n}-2.
	\end{align*}
\end{lemma}
\begin{proof}
	We further define the events 
	$$
	\tilE_{l}^{1}:=\bigg\{\max_{(i,j)\in J_{l}\setminus I_{l-1}}\left\vert\h A_{ij}\right\vert>\vep_{n}\bigg\} \quad \text{and} \quad  
	\tilE_{l}^{2}:=\bigg\{\max_{(i,j)\in J_{l}\setminus I_{l}}\left\vert\h A_{ij}\right\vert>\vep_{n}\bigg\}\,,
	$$
	where $$J_{l}:=\{(i,j):i\in [1,r_n],\, j\in [(l-1)r_{n}+1,lr_{n}] \}.$$
	Note that	
	 $\PP(F_{l})\leq P_m$ and $\PP(\tilE_{l}^{t})\leq P_{b}$ for $t\in\{1,2\}$.
	Splitting the event of two consecutive nonzero blocks according  to $F_{l}$ and $F_l^\cc$ gives
	$$
	\PP(E_{l}\cap E_{l+1})\le\PP(F_{l})+\PP(\tilE_{l}^{2}\cap\tilE_{l+1}^{1}\cap F_l^\cc)\le P_{m}+P_{b}^{2}.
	$$
	For three consecutive events, we split on both $F_{l}$ and $F_{l+1}$ and their complements to get the bound: 
	\begin{align*}
	\PP(E_{l}\cap E_{l+1}\cap E_{l+2}) 
	&\le \PP(F_{l}\cap F_{l+1}) +\PP(F_{l}\cap  F_{l+1}^\cc \cap E_{l+2})+\PP(E_{l}\cap F_l^\cc \cap F_{l+1})\\ &\qquad +\PP(\tilE_{l}^{2}\cap\tilE_{l+1}^{1}\cap\tilE_{l+2}^{1} \cap F_l^\cc \cap F_{l+1}^\cc)\\
	&\le P_{m}^{2}+2P_{m}P_{b}+P_{b}^3.
	\end{align*}
\end{proof}

\begin{lemma}\label{lem:3consecblocks}
	Denote by $\Omega_{n}^{(1)}$ the event that the first row of blocks in $\hA^{>\vep_{n}}$ has three (or more) consecutive nonzero blocks. If $\eta>2/5$, then $k_{n}\PP(\Omega_{n}^{(1)})\to 0$ as $n\to\infty$.
\end{lemma}
\begin{proof}	
	Using Lemma \ref{lem:consecblock}, $\PP(\Omega_{n}^{(1)})$ is bounded by
	\begin{align*}
	\sum_{1\leq l\leq k_{n}-2}\PP(E_{l}\cap E_{l+1}\cap E_{l+2})\leq k_{n}(P_{m}^{2}+2P_{m}P_{b}+P_{b}^{3}).
	\end{align*}
	By \eqref{lem:pbpm}, since $\beta>\frac{4}{3\alpha}$ and $\eta>2/5$, we have 
	\begin{align*}
	P_{m}^{2}+2P_{m}P_{b}+P_{b}^{3}&\le L(n)\bigg( \frac{4m^{2}r_{n}^{2}}{n^{4}\vep_{n}^{2\alpha}}+2\frac{4mr_{n}^{3}}{n^{4}\vep_{n}^{2\alpha}}+\frac{8r_{n}^{6}}{n^{6}\vep_{n}^{3\alpha}}\bigg)\\
	&\le O(L(n)n^{3-3\eta-8/3}+L(n)n^{6-6\eta-4}) \le o(k_{n}^{-1}).
	\end{align*}
	
\end{proof}

\begin{lemma}\label{lem:5blocks}
	Denote by $\Omega_{n}^{(2)}$ the event that the first row of blocks in ${\hA^{>\vep_n}}$ contains at least five nonzero blocks. If $\eta>5/6$, then
	$\lim_{\nto} k_n\PP(\Omega_{n}^{(2)})= 0$.
\end{lemma}
\begin{proof}
	We denote $\calC:=\{(l_{1},l_{2},l_{3},l_{4},l_{5})\,\vert\,\, \text{for}\,\,\,\, 1\le l_{1}<l_{2}<l_{3}<l_{4}<l_{5}\le k_n \}$ and consider the disjoint union $\cup_{i=0}^{4}\calC_{i}=\calC$ described as follows: 
	\begin{align*}
	\calC_{0}= &\{(l_{1},l_{2},l_{3},l_{4},l_{5})\VT  l_{j}+1<l_{j+1}\text{ for all }j\} \\
	\calC_{1}= &\{(l_{1},l_{2},l_{3},l_{4},l_{5})\VT l_{j}+1=l_{j+1} \text{ for a single }j\} \\ 
	\calC_{2}= &\{(l_{1},l_{2},l_{3},l_{4},l_{5})\VT l_{j}+1=l_{j+1} \text{ for exactly two $j$'s}\} \\ 
	\calC_{3}= &\{(l_{1},l_{2},l_{3},l_{4},l_{5})\VT l_{j}+1=l_{j+1} \text{ for exactly three $j$'s}\}\\ 
	\calC_{4}= &\{(l_{1},l_{2},l_{3},l_{4},l_{5})\VT l_{j}+1=l_{j+1} \text{ for exactly four $j$'s}\}. 
	\end{align*}
	Recalling the definition of $E_l$ in \eqref{eq:ElFl}, also denote  $\Omega_{n}^{(2)}(l_{1},l_{2},l_{3},l_{4},l_{5}):= \bigcap_{i=1}^{5} E_{l_{i}}.$ We will consider the union of $\Omega_{n}^{(2)}(l_{1},l_{2},l_{3},l_{4},l_{5})$ over the various  $\calC_{0},\cdots,\calC_{4}$, in order to bound $\PP(\Omega_{n}^{(2)}).$
	
	First observe that at least three of the $l_{i}$'s are consecutive for any $(l_{1},\cdots,l_{5})$ in $\calC_{3}\cup\calC_{4}$, so that 
	\begin{equation*}
	\bigcup_{(l_{1},\cdots,l_{5})\in\calC_{3}\cup\calC_{4}}\Omega_{n}^{(2)}(l_{1},\cdots,l_{5})\subset \Omega_{n}^{(1)}.
	\end{equation*}
	Thus by Lemma \ref{lem:3consecblocks}, it suffices to bound the probability of the union over $\calC_{0},\calC_{1}$ and $\calC_{2}$. 
	\begin{itemize} 
		\item [(0)] For $(l_{1},l_{2},l_{3},l_{4},l_{5})\in\calC_{0},$  all blocks are independent so that
		$$\PP(\Omega_{n}^{(2)}(l_{1},l_{2},l_{3},l_{4},l_{5}))\leq \prod_{i=1}^{5}\PP(E_{l_{i}})\leq P_{b}^{5}.$$
		
		\item [(1)] Let $(l_{1},l_{2},l_{3},l_{4},l_{5})\in\calC_{1}$. Without loss of generality, we may assume that 		
		$l_{1}+1=l_{2}$. Then Lemma \ref{lem:consecblock} implies 
		$$\PP(\Omega_{n}^{(2)}(l_{1},l_{2},l_{3},l_{4},l_{5}))\leq \PP(E_{l_{1}}\cap E_{l_{2}}) \prod_{i=3}^{5}\PP(E_{l_{i}})\leq(P_{m}+P_{b}^{2}) P_{b}^{3}.$$ 
		
		\item [(2)] Let $(l_{1},l_{2},l_{3},l_{4},l_{5})\in\calC_{2}$.
			Since we have already taken care of events where there are 3 consecutive blocks, without loss of generality, we may assume that $l_{1}+1=l_{2}$, $l_{2}+1<l_{3}$, and $l_{3}+1=l_{4}$. From the above bound for two consecutive blocks we have	$$\PP(\Omega_{n}^{(2)}(l_{1},l_{2},l_{3},l_{4},l_{5}))\leq \PP(E_{l_{1}}\cap E_{l_{2}})\PP(E_{l_{3}}\cap E_{l_{4}}) \PP(E_{l_{5}})\leq (P_{m}+P_{b}^{2})^{2}P_{b}.$$
	\end{itemize} 
	To get a union bound, we estimate the number of indices in each portion of the partitioned index set. Then, the sum can be bounded by    
	\begin{equation*}
	\begin{split}
	\sum_{\calC} \PP(\Omega_{n}^{(2)}(l_{1},l_{2},l_{3},l_{4},l_{5}))
	&\le
	k_{n}^{5}P_{b}^{5} +k_{n}^{4}P_{m}P_{b}^{3}+k_{n}^{3}P_{m}^{2}P_{b}\\
		&\le L(n)\bigg( k_{n}^{5}r_{n}^{10} n^{-5\alpha\beta}+ k_{n}^{4} r_{n}^{7} n^{-4\alpha\beta}+k_{n}^{3} r_{n}^{4} n^{-3\alpha\beta}\bigg).
	\end{split}
	\end{equation*}
	Since $\frac{4}{3\alpha}<\beta$ and $\eta>5/6$, the \rhs~is $o(k_{n}^{-1})$ as $n\to\infty$.
\end{proof}



Our next lemma says that each block, which survives the $\vep_n$ truncation, has a small number of nonzero entries which all lie in some $(2m+1)\times(2m+1)$ square.

\begin{lemma}\label{lem:2m+1entriesintheblock} Let $\eta>5/6$.
	Denote by $\Omega_{n}^{(3)}$ the event that the first row of blocks in $\hA^{>\vep_{n}}$ has a block such that all the nonzero entries in the block do not fit inside some $(2m+1)\times(2m+1)$ square. Then $k_{n}\PP(\Omega_{n}^{(3)}) \to 0$ as $n\to\infty$.
\end{lemma}	
\begin{proof}

	Defining 
\begin{equation*}
G_{l}:=\{\hB_{1l}^{>\vep_n} \text{ has nonzero entries that do not fit inside some square of size } (2m+1)\times (2m+1)\},
\end{equation*}
it suffices to prove that the following probability tends to zero:
\begin{equation*}
k_n\PP(\Omega_{n}^{(3)}\cap {\Omega_{n}^{(1)}}^{\cc}\cap {\Omega_{n}^{(2)}}^{\cc})\le k_{n}\sum_{l=1}^{k_n}\E(\indicator_{G_{l}}|{\Omega_{n}^{(1)}}^{\cc}\cap {\Omega_{n}^{(2)}}^{\cc})\, \PP({\Omega_{n}^{(1)}}^{\cc}\cap {\Omega_{n}^{(2)}}^{\cc}) .
\end{equation*} 
Note that the sum on the right-side, while over all $1\le l\le k_n$, has only four nonzero terms a.s. since $\indicator_{G_{l}}$ is a.s. zero for all but at most four of the $l$ indices. We next consider a given nonzero block $\hB_{1l}$ in $\hA^{>\vep_{n}}$ and denote by $(\h i ,\h j)$ the (in lexicographic order) minimal index of $\mathcal{O}$, where 
$$\mathcal{O}:=\{(i^{'},j^{'})\in \mathcal{I}_{l}\,\vert\, \h A^{>\vep_{n}}_{i^{'}j^{'}} =\max_{(i,j)\in \mathcal{I}_{l}} |\h A^{>\vep_{n}}_{ij}|\}$$
and $\mathcal{I}_{l}:=\{(i,j) \,:\, 1\le i\le {r_{n}}, (l-1)r_{n} {<} j\le lr_{n} \}.$
Consider the entries $\{\h A_{ij}\}_{(i,j)\in \tilde{\mathcal{I}}_{l}}$,  where $$\tilde{\mathcal{I}}_{l}:=\mathcal{I}_{l}\cap \{(i,j)\,:\, \max\{|i-\h i|,|j-\h j|\}>m \}.$$ 
It is important to note that after removal of entries within distance $m$ from positions $(\h i, \h j)$ the entries of $\hB_{1l}^{>\vep_n}$ are still identically distributed.
By $m$-dependence and stationarity, we have   
\begin{align*}
\PP\{\exists (i,j)\in  \tilde{\mathcal{I}}_{l}: |\h A_{ij}|> \vep_{n}\} &\le r_{n}^{2}\PP(|\h A_{ij}|> \vep_{n} )	\le \frac{L(n)r_{n}^{2}}{n^{2}\vep_{n}^{\alpha}}\,.
\end{align*}	
Putting things together we have 
\begin{equation*}
k_{n}\sum_{l=1}^{k_n}\E(\indicator_{G_{l}}|{\Omega_{n}^{(1)}}^{\cc}\cap {\Omega_{n}^{(2)}}^{\cc})\le L(n)4k_{n}r_{n}^{2}n^{-2}(\vep_{n})^{-\alpha}\le L(n)n^{2-\eta-\alpha\beta} \to 0, \qquad \nto\,.
\end{equation*}
\end{proof}

We will now introduce another decomposition so that we are left with at most one block whose  max norm is nonzero in any given row or column. To this end, set
\begin{align}\label{eq:2nd trunc}
{\hA^{(2)}}:={\hA^{>\vep_n}}-{\hA^{>\tvep_{n}}}.
\end{align}
We get the following result.
\begin{lemma}\label{lem:A(2)norm}
	If $\eta>5/6$, we have $\norm{\hA^{(2)}}\cip 0$ as $n\to\ff$.
\end{lemma}	
\begin{proof}
Extend the events	${\Omega_{n}^{(1)}}$, ${\Omega_{n}^{(2)}}$, and  ${\Omega_{n}^{(3)}}$ to each row $1\le i\le k_n$, and denote these events by ${\Omega_{n}^{(\cdot)}}(i)$. Define $$\Omega_{n}:=\bigcap_{i=1}^{k_n}\left({\Omega_{n}^{(1)}}^{\cc}(i)\cap {\Omega_{n}^{(2)}}^{\cc}(i)\cap {\Omega_{n}^{(3)}}^{\cc}(i)\right).$$   By Lemma \ref{lem:3consecblocks}, Lemma \ref{lem:5blocks} and Lemma \ref{lem:2m+1entriesintheblock}, the probability of $\Omega_n$ tends to 1 as $\nto$. On the event  $\Omega_{n}$, using the fact that the entries of $\hA^{(2)}$ are bounded by $\tvep_{n}$, we bound the spectral norm by the maximum row-sum norm to obtain
	\begin{align*}
		\|\hA^{(2)}\| &\le  4 (2m+1) \tvep_{n} \to 0\,, \qquad \nto\,.
	\end{align*}
\end{proof}



\subsubsection{Proof of Theorem \ref{tm:ReflMat} for $2\le \alpha<4$}
Choose $\eta>5/6$. We recall from \eqref{eq:conv to ppp} that
$\sum_{k=1}^{k_n}\sum_{l=k+1}^{k_n} \delta_{\hB_{kl}}  
\dto \sum_{i=1}^\infty\delta_{ (P_i\bsQ_{i}) }$, as $n \to \infty$.
Since $\tvep_{n}\to 0$ as $n\to\infty$, this implies
\begin{align}\label{eq:truncated matrix to ppp}
\sum_{k=1}^{k_n} \sum_{l=k+1}^{k_n}  \delta_{\hB_{kl}^{>\tvep_{n}}} \;  
\dto \quad \sum_{i=1}^\infty\delta_{ (P_i\bsQ_{i}) }\,,\quad \nto \;.
\end{align}
{Similar to the case $\alpha<2$, identify $\tilde{\ell}_{0,2m}$  with the space of  $(2m+1)\times (2m+1)$ matrices so that 
taking singular values of elements in $\tilde{l}_{0,2m}$ corresponds to taking
singular values in the space of $(2m+1)\times (2m+1)$ matrices.}

By Lemma \ref{lem:2m+1entriesintheblock}, we have that for any $\delta\in(0,1)$, there exists $N$ such that $n>N$ implies that for all $(k,l)$ we have
$\hB_{kl}^{>\tvep_n}$ is in 
$\tilde{l}_{0,2m}$ with probability $1-\delta$.
Moreover, $P_i\bsQ_{i}$ is also in 
	$\tilde{l}_{0,2m}$ for every $i$, a.s. (in fact, it is in 	$\tilde{l}_{0,m}$).

{Thus, as a marked point process we have
\begin{align*}
\sum_{k=1}^{k_n} {\sum_{l=k+1}^{k_n}} \sum_{j=1}^{2m+1}
{\delta_{\sigma_j(\hB_{kl}^{>\tvep_{n}})}} \;  
\dto \quad \sum_{i=1}^\infty \sum_{j=1}^{m+1} \delta_{ \sigma_j(P_i\bsQ_{i}) } \;.
\end{align*}
By Lemma \ref{lem:factor22}, with probability going to $1$, the positive eigenvalues of $\hA^{>\tvep_{n}}$ are the singular values of the $\hB_{kl}^{>\tvep_{n}}$. Hence, the above convergence also holds for the eigenvalues of $\hA^{>\tvep_{n}}$. It remains to show that the same convergence holds for the eigenvalues of $\hA$.} 
Similar to the case of $\alpha<2$, we employ Weyl's inequality to get
\begin{equation*}
\max_i |\lambda_{i} ({\hA}) - \lambda_{i} (\hA^{>\tvep_{n}})|
\leq \|{\hA^{<\vep_n}} \|+\|\hA^{(2)}\| \, ,
\end{equation*}
and by Proposition \ref{lem:aboundofallevenpaths} and Lemma \ref{lem:A(2)norm}, the right hand side goes to 0 in probability as $n\to\ff$.
For the same arguments given in the $\alpha<2$ case below Remark \ref{rem:beforeproof}, this concludes the proof of Theorem \ref{tm:ReflMat} for $2\le \alpha <4$.


\section{Proof of Theorem \ref{tm:CovMat}}\label{sec:proofcov}

We group the entries of the $p\times n$ matrix $\bsA$ into blocks of size
$r_n\times r_n$ (assuming $n = r_n  k_n$ and $p=r_n {\t k_n}$ without loss of generality). If $\bsB_{kl}$ denotes the $kl$ block of $\bsA$, one can check that the $kl$ block of $\bsA {\bsA}'$ takes the form
\[
\sum_{i=1}^{k_n} \bsB_{ki} {\bsB_{li}}'\,.
\]



\subsection{Sample covariance matrices : $0<\alpha<2$}
We use the truncation defined in \eqref{eq:truncate} 
and set 
\begin{equation*}
\ccC :=\bsA{\bsA}',\quad  \ccC_{1}:=\bsA^{<\varepsilon}{\bsA^{<\varepsilon}}' ,\quad\text{and}\quad \ccC_{2}:=\ccC-\ccC_{1}.
\end{equation*}
Recalling \eqref{eq:frob+markov}, for any $\delta>0$, we have $$\lim_{\vep\to 0}\limsup_{n\to\infty}\PP\left(\|\ccC_{1}\|>\delta\right)= 0.$$ We observe that 
$\ccC_{2}=\bsA^{<\varepsilon}{\bsA^{>\varepsilon}}'+\bsA^{>\varepsilon}{\bsA^{<\varepsilon}}'+\bsA^{>\varepsilon}{\bsA^{>\varepsilon}}'$.
Let  us consider the event that $\bsA^{>\varepsilon}$ has all diagonal blocks equal to zero and at most one nonzero block in each row and column. By the argument of Lemma \ref{lem:En},  the probability of that event goes to $1$. On this event, the diagonal entries of $\bsA^{<\varepsilon}{\bsA^{>\varepsilon}}'$ are also zero. We next show that the off-diagonal blocks of $\bsA^{<\varepsilon}{\bsA^{>\varepsilon}}'$  are negligible with probability going to one, as $\varepsilon$ goes to zero. 

\begin{lemma} \label{lem:offdiagonal} For any $\delta>0,$ 
	\begin{equation}\label{eq:offdiagonal}
	\lim_{\vep\to 0}\limsup_{n\to\infty}\PP\left(\| \bsA^{<\varepsilon}{\bsA^{>\varepsilon}}'\|>\delta\right)= 0.
	\end{equation}
\end{lemma}
\begin{proof}
	Define $(i^*,j^*)$ such that the maximal entry of $\bsA^{>\varepsilon}$, in absolute value, is attained in the block $\bsB_{i^{*}j^{*}}$.
	For any $M>0$,
	we have
	\begin{align*}
	\lim_{\vep\to 0}\limsup_{n\to\infty}\PP{\left(\|\bsA^{<\varepsilon}{\bsA^{>\varepsilon}}'\|>\delta\right)}
	&\leq \lim_{\vep\to 0}\limsup_{n\to\infty}\PP\left(\|\bsA^{<\varepsilon}\|\|\bsB_{i^{*}j^{*}}\|>\delta\right)\\
	&\leq\lim_{\vep\to 0}\limsup_{n\to\infty} \PP(\|\bsA^{<\vep}\|>\frac{\delta}{M})+\PP(\|\bsB_{i^{*}j^{*}}\|>M)\\
	&\leq\limsup_{n\to\infty}\PP(\|\bsB_{i^{*}j^{*}}\|>M).
	\end{align*}
	But now, the right side goes to $0$ as $M\to\infty$.
\end{proof}
Weyl's inequality yields that
\begin{equation*}
\max_{i}\left|\lambda_{i}\left(\bsA{\bsA}'\right)-\lambda_{i}\left(\bsA^{>\vep}{\bsA^{>\vep}}'\right)\right|\leq \|\bsA^{<\vep}{\bsA^{<\vep}}'\|+\|\bsA^{<\vep}{\bsA^{>\vep}}'\|+\|\bsA^{>\vep}{\bsA^{<\vep}}'\| \to 0.
\end{equation*} 
Thus, the eigenvalues of $\bsA{\bsA}'$ are asymptotically determined by the non-zero diagonal blocks of $\bsA^{>\vep}{\bsA^{>\vep}}'$ which are, with probability going to one, blocks of the form  $\bsB_{kl}^{>\vep} {\bsB_{kl}^{>\vep}}'$ where $\bsB_{kl}^{>\vep}$ is a non-zero block of $\bsA^{>\vep}$. Repeating the arguments below Remark \ref{rem:beforeproof} concludes the proof in the sample covariance setting with $\alpha\in (0,2)$. 

\subsection{Sample covariance matrices : $2\leq\alpha<4$ }
Similar to the case $0<\alpha<2$, we set
\begin{equation*}
\ccC =\bsA{\bsA}',\quad  \ccC_{1}=\bsA^{<\varepsilon_{n}}{\bsA^{<\varepsilon_{n}}}' ,\quad\text{and}\quad \ccC_{2}=\ccC-\ccC_{1}
\end{equation*}
where $\varepsilon_{n}=\frac{n^{\beta}}{a_{np}}$ with $\beta$ satisfying \eqref{eq:betabounds}. 
First, we show that $\ccC_1$ does not contribute to the spectrum in the limit.

\begin{lemma}
	\label{lem:2<alpha<4offdiagonal2} For any $\delta>0,$ 
	\begin{equation}\label{eq:2<alpha<4offdiagonal2}
	\lim_{n\toi}\PP\left(\|\ccC_1\|>\delta\right)= 0.
	\end{equation}
\end{lemma}
\begin{proof}
	Following the argument from \cite[p.609]{auffinger:arous:peche:2009}, for truncated matrices, we have 
	\begin{equation*}
	\EE \left( \mathrm{Tr} \hspace{1pt} \(\frac{\bsX^{<}{\bsX^{<}}'}{n^{4/\alpha-\vep}} \)^{s_n}   \right)=\frac{1}{n^{s_n(4/\alpha-\vep)}} \sum_{\mathfrak P} \EE  X^{<}_{i_1i_0} X^{<}_{i_1i_2} X^{<}_{i_3i_2}X^{<}_{i_3i_4} \cdots  X^{<}_{i_{2s_n-1}i_{2s_n-2}} X^{<}_{i_{2s_n-1}i_0}\,,
	\end{equation*}
	where $\mathfrak{P}$ denotes the set of all closed paths $\mathcal{P}=\{i_0,i_1,\dots ,i_{2s_n-1},i_0\}$ with a distinguished origin.
	Recalling the number of maps from $\mathcal{P}$ to $\mathcal{P}^{\prime}$  under $\psi$, 
	\begin{align*}
	\EE \left( \mathrm{Tr} \hspace{1pt} \(\frac{\bsX^{<}{\bsX^{<}}'}{n^{4/\alpha-\vep}} \)^{2s_n}   \right) &\leq 2 (m+1)^{4s_n}(4s_n)! \sum_{\mathfrak{P}_{\text{even}}} \EE\left(\frac{1}{n^{2s_n(4/\alpha-\vep)}}  \bsY^{<}{\bsY^{<}}'({\mathcal P})\right)\\
	& \leq o(1) 2 (m+1)^{4s_n} 4^{4s_n} (2\e)^{2s_n}\frac{(4s_n)!}{(2s_n)! (2s_n+1)!}n\,,
	\end{align*}
	where the last inequality is given by the proof of Proposition \ref{lem:aboundofallevenpaths} and Proposition 20 in \cite{auffinger:arous:peche:2009}. Using the { Stirling formula}, for any $\gamma>0$  there exists $C>0$ such that
	\begin{align*}
	\PP&\left(  \lambda_{1} \(\frac{1}{n^{4/\alpha-\vep}}  \bsX^{<}{\bsX^{<}}' \) >  2 e(m+1)^2 (8+\gamma)^2\right) \\
	&\leq  \frac{\EE \left( \lambda_{1}\(\frac{1}{n^{4/\alpha-\vep}}  \bsX^{<}{\bsX^{<}}'\) \right)^{2s_n}}{ (2 \e(m+1)^2 (8+\gamma)^2)^{2s_n} } 
	\leq  \frac{\EE \left( \mathrm{Tr} \hspace{1pt} \left( \lambda_{1}\(\frac{1}{n^{4/\alpha-\vep}}  \bsX^{<}{\bsX^{<}}'\) \right)^{2s_n}  \right)}{ (2 \e(m+1)^2 (8+\gamma)^2)^{2s_n} } 
	\leq \exp{(-Cs_n)}. 
	\end{align*}
	Moreover, since $\frac{n^{2/\alpha-\epsilon}}{a_{np}}\to 0$ as $n\toi$, the above inequality implies that $\|\bsA^{<\vep_{n}}{\bsA^{<\vep_{n}}}'\|\to 0$ in probability.  
\end{proof}

Next, we analyze the remaining $\ccC_{2}$. As in \eqref{eq:2nd trunc}, we use the decomposition 
${\bsA^{>\vep_n}}={\bsA^{>\tvep_{n}}}+{\bsA^{(2)}}$. 
Then,
\begin{align}\label{c decomp}
\ccC_{2}
&=\bsA^{<\varepsilon_{n}}{\bsA^{>\varepsilon_{n}}}'+\bsA^{>\varepsilon_{n}}{\bsA^{<\varepsilon_{n}}}'+\bsA^{>\varepsilon_{n}}{\bsA^{>\varepsilon_{n}}}'\\
&\nn =\bsA^{<\varepsilon_{n}}{\left({\bsA^{>\tvep_{n}}}+{\bsA^{(2)}}\right)}'+\left({\bsA^{>\tvep_{n}}}+{\bsA^{(2)}}\right){\bsA^{<\varepsilon_{n}}}'+\left({\bsA^{>\tvep_{n}}}+{\bsA^{(2)}}\right)\left({\bsA^{>\tvep_{n}}}+{\bsA^{(2)}}\right)'.
\end{align}
From the proof of Lemma \ref{lem:2blocks} we know
 that at most one block is nonzero, in any given row of      $\bsA^{>\tvep_{n}}$, with probability going to one as $n\to\ff$. 
By the argument in the proof of Lemma \ref{lem:offdiagonal}, $\Vert\bsA^{>\tvep_{n}}{\bsA^{<\vep_{n}}}'\Vert \cip 0$. An application of Lemma \ref{lem:A(2)norm} shows that  $\Vert\bsA^{(2)}{\bsA^{<\tvep_{n}}}'\Vert$ and $\Vert\bsA^{(2)}{\bsA^{(2)}}'\Vert$ tend to $0$, in probability as $n\to\ff$. Moreover, from those three lemmas, we conclude that $\Vert\bsA^{>\tvep_{n}}{\bsA^{(2)}}'\Vert \to 0$  in probability, as $n\to\ff$, as well.

Using the decomposition in \eqref{c decomp} and Weyl's inequality, we have 
\begin{align*}
\max_{i=1,\ldots,p} \Big|\lambda_{i}\left(\bsA{\bsA}'\right) & -\lambda_{i}\left(\bsA^{>\tvep_{n}}{\bsA^{>\tvep_{n}}}'\right)\Big|
\leq \norm{\bsA^{<\vep_{n}}{\bsA^{<\vep_{n}}}'}+\norm{\bsA^{<\vep_{n}}{\bsA^{>\vep_{n}}}'}+\norm{\bsA^{>\vep_{n}}{\bsA^{<\vep_{n}}}'}\\
&\quad+\norm{\bsA^{(2)}{\bsA^{(2)}}'}+\norm{\bsA^{>\tvep_{n}}{\bsA^{(2)}}'}+\norm{\bsA^{(2)}{\bsA^{>\tvep_{n}}}'} \cip 0\,, \quad \nto\,.
\end{align*}
Again, repeating the arguments below Remark \ref{rem:beforeproof} concludes the proof in the sample covariance setting with $2\le \alpha<4$. 

\section*{Acknowledgements}   The work of B.~Basrak has been supported in part by {the SNSF/HRZZ Grant CSRP 2018-01-180549}.  The work of Y.~Cho and   P.~Jung was funded in part by the National
Research Foundation of Korea grant NRF-2017R1A2B2001952. J.~Heiny was supported by the Deutsche Forschungsgemeinschaft (DFG) via RTG 2131 High-dimensional Phenomena in Probability – Fluctuations and Discontinuity. J.H. thanks the Mathematics Department at KAIST for the hospitality.

\bibliographystyle{alpha}
\bibliography{bib2}
\end{document}